 		\documentclass[11pt, twoside, a4paper]{article}
		\usepackage{amsmath,amssymb,amsthm,amscd,appendix,calrsfs,cite,color,epsfig,eucal,enumitem,dsfont,graphics,graphicx,latexsym,mathrsfs,verbatim,hyperref}
\usepackage[english]{babel}

\usepackage{bbm}
\usepackage{amsmath}
\usepackage{amssymb}
\usepackage{amsfonts}
\usepackage{array}
\usepackage{lscape}
\usepackage{rotating}
\usepackage{hyperref}			
\usepackage{amsthm}

\newtheorem{teo}{Theorem}
\newtheorem{prop}{Proposition}
\newtheorem{cor}{Corollary}
\newtheorem{lem}{Lemma}

\newtheorem{rem}{Remark}

\newcommand{\indicator}[1]{\mathbbm{1}_{#1}}
\newcommand{\sms}{\setminus}
\newcommand{\N}{\mathbb{N}}
\newcommand{\Z}{\mathbb{Z}}
\newcommand{\C}{\mathbb{C}}
\newcommand{\R}{\mathbb{R}}
\newcommand{\D}{\mathbb{D}}
\newcommand{\B}{\mathbb{B}}

\newcommand{\X}{\overline{X}}

\newcommand{\zul}{\underline{0}}
\newcommand{\aul}{\underline{a}}
\newcommand{\bul}{\underline{b}}
\newcommand{\cul}{\underline{c}}
\newcommand{\dul}{\underline{d}}

\newcommand{\sul}{\underline{s}}
\newcommand{\tul}{\underline{t}}

\newcommand{\wul}{\underline{w}}
\newcommand{\uno}{\mathds{1}}

\newcommand{\eul}{\underline{e}}

 \DeclareMathOperator{\re}{Re}
 \DeclareMathOperator{\im}{Im}

\DeclareMathOperator{\ES}{ES}

\DeclareMathOperator{\Cb}{CB}
\DeclareMathOperator{\ExpShr}{Exp}
\DeclareMathOperator{\Conn}{Con}
\DeclareMathOperator{\Delt}{Delta}
\DeclareMathOperator{\Path}{Path}
 \DeclareMathOperator{\BV}{BV}
  
\DeclareMathOperator{\ULH}{ULH} \DeclareMathOperator{\Top}{top}
\DeclareMathOperator{\li}{L} \DeclareMathOperator{\Var}{Var}
\DeclareMathOperator{\Pe}{P}
 \DeclareMathOperator{\Sing}{Sing}
\DeclareMathOperator{\rad}{rad}

 \DeclareMathOperator{\ie}{i}
  \DeclareMathOperator{\iie}{ii}
   \DeclareMathOperator{\iiie}{iii}
    \DeclareMathOperator{\ive}{iv}
     \DeclareMathOperator{\ve}{v}
\title{On the existence of a conformal and an absolutely continuous invariant measure for transcendental entire maps}
\author{Irene Inoquio-Renteria \\ 
        \small{ICFM Universidad Austral de Chile, casilla 567 Valdivia, Chile} \\ \small{(email: ireneinoquio@uach.cl)} }

\begin{document}

\maketitle

\begin{abstract}
We identify  a class of hyperbolic transcendental entire maps  
and we prove that some of its elements  generate a class of potentials for which exhibit
 a conformal and invariant probability Gibbs measure. 
The methods  and techniques from the thermodynamic formalism can be extended to this class of potentials.
To complement this study
 we highlight that the dynamics of such a map  on some subset of the Julia set  is conjugated to the shift map over a code space with  countable alphabet and the euclidean metric on 
  the complex plane  induces a metric on the symbolic space  which is not compatible with the shift standard metric.  From this fact, 
 we provide a general description of the thermodynamic formalism from   symbolic dynamic outlook, by  studying  
 the shift map acting on a
non-compact and invariant subset 
 of the full shift space with a countably infinite alphabet and  a class of  weakly 
H\"older continuous potentials, to prove the
existence of a conformal and absolutely continuous invariant probability measure.
 
\end{abstract}
\maketitle

\section{Introduction}
Great deal of attention has been paid to the study of the
thermodynamic formalism of certain transcendental maps. 
In particular, the ergodic theory of the exponential family $E_{\lambda}(z) =\lambda\exp(z)$  has been described
in great detail for a large class of parameters, see ~\cite{McM87,St91a,Ka99,UZ03,UZ04,IS06,Ba07,BK07,CS07,MU10} and references therein.

When we deal with transcendental entire   function, some difficulties arise in  the study of the thermodynamic formalism, 
 on the one hand the Julia set
 is never  compact, 
 so problems of convergence  appear and the standard arguments such as Schauder-Tychonoff Fixed Point Theorem can not be applied. On the other  hand,  the 
 dynamics restricted to the Julia set is not Markov, so unfortunately  the symbolic spaces we need to model Julia sets of these maps
  fall out of  the  framework  developed by Sarig in~\cite{Sa99}  and 
  Mauldin-Urba\'nski~\cite{MU01}, who developed the thermodynamic formalism  for topologically mixing  Markov shift spaces  with infinitely many symbols.

 In the present work, we highlight  a class  of transcendental entire maps (See Defintion in~\S~\ref{Transcendental dynamics}) that includes the exponential family, and we prove that there is a hyperbolic transcendental entire map, which generates a large class of potentials  for which the existence of a conformal measure is guaranteed.  
The main novelty in this paper is the identification of a class of potentials  which is different from the ones studied earlier in the literature by, for example, Mayer and
Urba\'nski \cite{MU10} in the  study of  the thermodynamic formalism (see~Theorem~\ref{t:main}).

To develop a thermodynamic formalism of these maps with these associated potentials (see Theorem~\ref{teo4}) the techniques used in ~\cite{MU10} can be applied without making many changes, however it is worth addressing an alternative approach taking advantage of the properties of their symbolic representation of these maps acting on invariant subsets of the Julia sets, whose code spaces have a natural topology that is inherited from the Euclidean topology.
So, the second  part of this work is addressed to obtain  results from a symbolic dynamics setting, by means of  
  an approximation argument, considering  restrictions of the map to subsets of the Julia set that can be encoded by a full shift on $2N+1$ symbols, with $N$ as large as desired and since that space is compact and Markov the usual results hold and a conformal measure  $\nu_N$ exists. So, the  next step  is to show that the sequence $\{\nu_N\}_N$ is tight and so this sequence  has an accumulation point $\nu$. So  the measure $\nu$  would be  a conformal measure for the original map.  All of these  results 
 are proved for  the shift map acting on  some 
   non-compact  and invariant subset  of the full shift space  with a countably infinite number of symbols,  equipped  with a metric that 
    is not necessarily equivalent to the natural shift metric, therefore for a weakly H\"older continuous potentials  the existence of a  conformal  and absolutely continuous invariant probability measure are proved. 
 To do this, 
 we impose some additional conditions on the symbolic space and potentials 
 that are derived from the inherent properties of the transcendental dynamics.
 See~Theorem~\ref{teo1}, Theorem~\ref{teo2},  Theorem~\ref{teo3}.
 
 \medskip

 The paper is organized as follows. In~\S~\ref{Transcendental dynamics} we define the class of transcendental entire maps and  a class of potentials  
 to state and prove the Theorem~\ref{t:main}.
 After gather several properties of the dynamics  in Lemma~\ref{Dynamics Lemma} and properties of potentials in Proposition~\ref{pro: potentials}, the proof of Theorem \ref{t:main} follows from techniques in~\cite{MU10}.
After giving several metric conditions  on the symbolic space, some preliminary Lemmas of the shift map and potentials in \S~\ref{Appendix}, we prove the respective Theorem~\ref{teo1}, Theorem~\ref{teo2} and Theorem~\ref{teo3}.

\subsection{Hyperbolic transcendental entire maps}\label{Transcendental dynamics}
Given a transcendental entire function $f:\C\rightarrow \C$, the Fatou set $F(f)$ is the subset of $\C$ where the iterates $f^n$ of $f$ form a normal family and its complement is namely called the Julia set, which is  denoted by $J(f)$. 

Denote by $\Sing(f^{-1})$ the set of finite singularities of  the inverse function~$f^{-1},$ which is  the set of critical values (images of critical points) and asymptotic values of $f$ together with  their  finite limit points. 
 The \emph{post-singular} set  $PS(f)$ of $f$ is defined as, 
$$PS(f):= \bigcup_{n=0}^{\infty} f^n\left(\Sing\left(f^{-1}\right)\right),$$
 and  $\rho_f:=\displaystyle\limsup_{z\rightarrow\infty}\frac{\log\log |f(z)|}{\log|z|}$ is namely called \emph{the order of }$f$.

\medskip

Let   $\mathscr{F}$  denote  the class of transcendental entire functions $f$ satisfying the following properties
\begin{enumerate}
\item[1.] It is of \emph{finite order}; \label{1a}
\item[2.] satisfies the \emph{rapid derivative
growth condition}: There are $\alpha_2>0
 $ $\alpha_1>\alpha_2$ and $\kappa >0$  such that for every $z\in J(f)\backslash f^{-1}(\infty)$  we have $$ |f'(z)|\geq \kappa^{-1}|z|^{\alpha_1}|f(z)|^{\alpha_2};$$
\item[3.] it is of \emph{Disjoint type}, that is, the set $\Sing(f^{-1})$ is contained in a compact subset of the immediate  basin $B=B(z_0)$  of  an attracting fixed point $z_0\in \C$.
(See~\cite{BJR12}) 
 \label{4a}
 \end{enumerate}
 \medskip
  Note that  each $f\in \mathscr{F}$ belongs to the \emph{Eremenko-Lyubich} class 
$$
\mathcal{B}:=\{f:\C\rightarrow \C:\Sing(f^{-1}) \textrm{ is bounded}\}.
$$
It was proved in~\cite{Er92} that for $f\in \mathcal{B}$ all the Fatou components of $f$ are simply connected. Hence the immediate basin $B$ is simply connected. Moreover each $f\in \mathscr{F}$ is hyperbolic
 in the sense that the closure  of
$\overline{PS(f)}$ is disjoint from the Julia set  and 
$\overline{PS(f)}$ is compact. We have that  $f$ has no wandering  and Baker domains, so $B$ is the only Fatou component of $f$, see~\cite{Er89,Er92,GK86}.
\medskip 

  Examples in the class $\mathscr{F}$ include the family
$\lambda \exp(z)$ for $\lambda\in(0,1/e)$, the family of maps $\lambda \sin(z)$ for $\lambda\in(0,1),$ and 
 $\lambda g(z)$, where $\lambda\in\C\backslash\{0\}$  and
$g$ is an arbitrary  map of finite order such that 
$\Sing(g^{-1})$ 
is bounded and $\vert \lambda\vert$ is enough small, other examples are  the expanding entire maps 
 $\sum_{j=0}^{p+q} a_j e^{(j-p)z},$
  $p,q>0$, $a_j\in\C$, studied early in~\cite{CS07}.

  \medskip
 

\subsubsection*{Potentials}
Fix $f\in \mathscr{F}$. 
Since the immediate attraction basin $B=B(z_0)$ of an 
 attracting  fixed point $z_0$ is simply connected, there exists
 a bounded simply connected domain $D\subset \C$, such that  its closure  $\overline{D}\subset B$, and boundary $\partial  D$ is an
 analytic  Jordan  curve. Moreover,  $\overline{\Sing(f^{-1})}\subset  D$
 and $\overline{f(D)}\subset D$, for more details see~\cite[Lemma~3.1]{BK07}. 
 Following~\cite{BK07}, the pre-images of
$\C\backslash \overline{D}$ by $f$ consists of countably many unbounded connected components called  
   \emph{tracts} of $f$.
    We denote the collection of all these tracts by 
 $\mathscr{R}$.

 Since the closure of  each tract is simply connected, there exists
an open  simple arc
$\alpha:(0,\infty)\rightarrow \C\backslash \overline{D},$ which is 
disjoint from the union  of  the  closures of all tracts  and such that
 $\alpha(t)$ tends to a point
  of $\partial D$ as $t$ tends to $0^+,$ and $\alpha(t)$ tends to $\infty$ as $t$ tends  to $+\infty.$ We use this curve to define the fundamental domains on each tract as follows: since
  for every $T\in \mathscr{R}$ the map $f|_{T}$ is a cover of $\C\backslash \overline{D}$, we have 
 $T\backslash f^{-1}(\alpha)$ is the union of infinitely  many disjoint simply connected domains $S$ such that  the function
 $$f|_{S}:S\rightarrow\C\backslash (\overline{D}\cup \alpha)$$ is
 bijective. 
  Given $T\in\mathscr{R}$, we denote  by $\mathcal{S}_T$ the
collection of connected components of $T\backslash
f^{-1}(\alpha)$. The elements of
\begin{equation}\label{S}
\mathcal{S}:=\bigcup_{T\in\mathscr{R}}\mathcal{S}_T
\end{equation}
are 
called fundamental domains.

\vspace{0.5cm}

 \noindent For each $S\in \mathcal{S}$, we have that the restriction $f\vert_S$ is univalent, so we denote its inverse  branch by
  $
g_S:=(f|_{S})^{-1}: \C\backslash (\overline{D}\cup
\alpha)\rightarrow  S.
 $
 For  $n\geq 1$ and each $j\in \{0,1,\cdots, n\}$  denote by $S_j$
an element of $\mathcal{S}$
  and put
$g_{_{S_0S_1\cdots S_n}}=~
 g_{_{S_0}}\circ\cdots\circ
g_{_{S_n}}.$
 Then,
   \begin{align}\label{equ_g}
      g_{_{S_0 \cdots S_n}}(\C\backslash (\overline{D}\cup
\alpha))&=\left\{z\in \C: f^j(z)\in S_j, \text{ for every }
j=0,\cdots,
    n\right\}.  
        \end{align}
 For each sequence $\underline{S}=(S_0 S_1\cdots)\in \mathcal{S}^\N$, let
 $\displaystyle K_{\underline{S}}:=\bigcap_{n=0}^{\infty} g_{_{S_0 S_1\cdots S_n}}(\C\backslash (\overline{D}\cup~
\alpha)).$
   Then,
the Julia set of $f$ is given  by  the disjoint union of $K_{\underline{S}} $, that is  
$$ J(f)=~\displaystyle \bigsqcup_{\underline{S}\in \mathcal{S}^{\N}}
   K_{\underline{S}}.$$
   
 Since  
 $f$ has finite order and of disjoint-type, following 
 \cite{RRRS11},
 the Julia set $J(f)$ is a Cantor bouquet, that is a union of uncountably many pairwise
disjoint curves tending to infinity (hair) and each curve is attached to the unique  point accessible from  the immediate basin $B$, called endpoint of the hair. More precisely,
either
$K_{\underline{S}}$ is empty or there is a homeomorphism
$h_{\underline{S}}:[0,+\infty)\rightarrow K_{\underline{S}}$ such
that $\lim_{t\rightarrow+\infty}h_{\underline{S}}(t)=~\infty$, and
such that  for every  $t>0 $ we have
$\lim_{n\rightarrow+\infty}f^n(h_{\underline{S}}(t))=\infty.$ In
the latter case $z_{\underline{S}}:= h_{\underline{S}}(0)$ is the
only point of $K_{\underline{S}}$ accessible~\footnote{ If  $U$ is simply connected domain in the Riemann sphere $\overline{\C}$, we say that a point
 $z\in \partial U$ is \emph{accessible} from~$U$ if there
 exists  a  curve $\upsilon:[0,\infty)\rightarrow U$ such
 that $\lim_{t\rightarrow \infty}\upsilon(t)=z$.} from the immediate
basin $B$.  See also~\cite{Ba07}, 
which  generalizes previous results
for the exponential map having an attracting fixed point of~\cite{DG87}.

\medskip



Let $\rho_f$ be the order of $f$ and  
 $\alpha_1,\alpha_2>0$ be the corresponding
constants of the rapid derivative
growth condition of $f$.
 Fix $\tau\in (0,\alpha_2)$ and let
$\gamma: \C\backslash\{0\}\rightarrow \R\cup\{\infty\}$  defined
by $\gamma(z)=\frac{1}{|z|^\tau}$. Let 
   $\theta$ be  the   Riemaniann metric
 on $\C\backslash\{0\}$ defined by
 $$d\theta(z)=\gamma(z)|dz|,$$
 and we  derive $f$  with respect to  $\theta$
  instead of the Euclidean  metric. So,  for each $z\in \C\backslash\{0\}$  we have
 \begin{equation}\label{eq:derivada}
|f'(z)|_{\theta}= |f'(z)|\frac{\gamma\circ f(z)}{\gamma(z)}=
 |f'(z)|\frac{|z|^\tau}{|f(z)|^\tau}.
\end{equation}
 Denote by  $\mathscr{C}$ be the set of
 functions  $\psi$ from
$\bigcup_{S\in\mathcal{S}}S$ to $\R^+$ that are bounded   from
above and are
  constant  on each element of $\mathcal{S}$. 
  
  \begin{equation}
  \mathscr{C}:=\left\{\psi:\bigcup_{S\in\mathcal{S}}S\rightarrow \R^+: \psi \textrm{ is bounded from above and constant over  each }  S\in \mathcal{S}\right \}
  \end{equation}

\medskip 

Let us define the following class of potentials
 for $f$:
\begin{equation*}
\mathscr{P}_f=\left\{\varphi_{\psi,t}(z)=\log
\psi(z)-t\log|f'(z)|_\theta,  \psi\in \mathscr{C},
t>\frac{\rho_f}{\alpha_1+\tau}\right\}.
\end{equation*}
Observe that this class contains potentials  $-t\log\vert f'\vert_\theta$, which  from~(\ref{eq:derivada}) are cohomologous to $-t\log \vert f'\vert$.
\vspace{0.2cm}

For each $f\in \mathscr{F} $ we denote  by $\mathscr{T}_f$ the class of tame potentials stated in~\cite{MU10}, that is
$$
\mathscr{T}_f:=\left\{\varphi= h-t\log \vert f'\vert_{\theta};\, h \textrm{ is bounded weakly H\"older function, }t>\frac{\varrho'}{\alpha_1+\alpha_2}\right\}
$$

Although the class $\mathscr{F}$ does not include most of the functions  
 considered in~\cite{MU10},  our first result highlights that 
the class  of potentials $\mathscr{P}_f$ determined by  $f\in \mathscr{F}$ intersects the class of potentials  considered in~\cite{MU10},  which the difference is non-empty. So the first result is the following
\begin{teo}\label{t:main}
There exists $f\in \mathscr{F}$ such that $\mathscr{P}_f\cap \mathscr{T}_f\neq \emptyset$ and $\mathscr{P}_f\backslash \mathscr{T}_f\neq \emptyset$.
\end{teo}




\subsubsection*{Symbolic representation }

Let 
    $\Sigma_{\Z}:=\{\sul=(s_0s_1\ldots):s_j\in\Z, \text{ for all } j\geq 0\}$
   be the full shift space, and the shift metric is defined as follows, for    
   for some $\theta\in(0,1)$, 
\begin{equation}\label{metri}
d(\sul,\tul)=\theta^{\inf\{k: s_k\neq t_k\}\cup\{\infty\}}.
\end{equation} 
For every $n\ge 1$, we denote a
finite word $s_0\cdots s_{n-1}$ in~$\Z^n$  simply by $s^*$, so we follows the following notation for cylinders
$$[s^*]=\{\underline{w}\in \Sigma_{\Z}:w_i=s_i, 0\leq i\leq n-1\}$$
  and for $s\in\Z$, we simply denote
$[s]=\{\underline{w}\in \Sigma_{\Z}:w_0=s\}.$ Let  $\sigma:\Sigma_{\Z}\longrightarrow \Sigma_{\Z}$ be the left-sided shift map, given by  $\sigma(s_0s_1\cdots)=(s_1s_2\cdots).$

Observe that by definition the set $\mathcal{S}$ given in (\ref{S}) is countably infinite, so we identify
$\mathcal{S}$ with $\Z$. Put
\begin{equation}\label{eq:the set X}
X:=\{\underline{S}\in \mathcal{S}^\N:
K_{\underline{S}}\neq\emptyset\}\subseteq \Sigma_{\Z}.
\end{equation}
Let
$$
 Z=\bigcup_{\underline{S}\in X}
 K_{\underline{S}}.$$
  From~(\ref{equ_g}), we have  for each $\underline{S}\in \mathcal{S}^\N$, 
 $f(K_{\underline{S}})= K_{\sigma(\underline{S})}$, then   the function $f$ on the Julia set $J(f)$  is semi-conjugate  to
  $\sigma$ on  $X$, however $f$ on the set
       $$\mathcal{EP}:=\{z_{\underline{S}}=
  h_{\underline{S}}(0):
\underline{S}\in X\},$$
 is
conjugate to $\sigma\vert_X$. Hence the set  $X $ is completely
$\sigma$-invariant.
\medskip

The set  $\mathcal{EP}$ defined above, is the set of \emph{endpoints} of  hairs $K_{\underline S}$ and it 
satisfies the following properties, it is the set of accessible points from the immediate attraction basin $B.$ It  is
totally disconnected,  however $\mathcal{EP}\cup\{\infty\}$ is connected, 
see~\cite{BJR12}
Moreover 
 following~\cite{Ba08},  the  Hausdorff dimension of this set is equal to two, generalizing  previous results of Karpi\'nska~\cite{Ka99} for the exponential map  $f_\lambda(z)=\lambda e^z$ with  parameters $\lambda\in (0,1/e) $. This 
 exponential map is probably the best known example in the family $\mathscr{F}$, its Julia set is a 
Cantor bouquet 
 and the set of endpoints is modelled by the symbolic space of all allowable sequences, see~\cite{DK84} and~\cite{DG87}.

\medskip

In the following, we state  some properties concerning to the dynamics $(X, \sigma\vert_X)$, endowed with a metric  inherited from the euclidean metric
on $J(f)$.
It does
not necessarily generate the topology induced by the cylinder sets. 
 
Let $H: X\times[0,+\infty)\rightarrow J(f),$ and
$H|_{X\times\{0\}}: X\times\{0\}\rightarrow \mathcal{EP}$
defined by $H(\sul,0)=h_{\sul}(0)$, we have  that $H$ induces a  metric $\varrho$  on $X$ 
$$\varrho(\sul,\wul):=\vert h_{\sul}(0)-h_{\wul}(0)\vert.$$
  The shift map $\sigma:X\rightarrow X$ is continuous respect to $\varrho$.

\medskip

  Given $\sul\in\Sigma_{\Z}$  and $w^*\in~\Z^n$ let us write
 $
w^*\sul=(w_0\cdots w_{n-1}s_0 s_1\cdots).
 $
 For  a  set $A\subset \Sigma_{\Z}$ write $w^*A=\left\{w^*\sul: \sul\in
 A\right\}$. For $\sul\in X$ and $\delta>0 $ we define  the following sets with respect to the metric $\varrho.$
$$B(\sul,\delta):=\big \{\wul\in X:\varrho(\sul,\wul)<\delta\}=\big \{\wul\in X:\vert h_{\sul}(0)-h_{\wul}(0)\vert<\delta\big \},$$
$$  
\overline{B}(\zul,\delta):=\{\sul\in X: \varrho(\sul,\zul)\leq \delta\}.$$
  
    $$\B_0(\sul,\delta):=\big \{\wul\in X:\varrho(\sul,\wul)<\delta \& b_0~=~a_0\big \}.$$
     For every $n\geq 1$ and  $\sul\in X$ define
  
$$
\B_n(\sul,\delta):=\left \{\wul\in X:\sigma^j(\wul)\in
 \B_0(\sigma^j(\sul),\delta), \textrm{ for  all }
j=0,1,\cdots,n\,
 \right\}.$$



\medskip

The set $X$
endowed with the metric $\varrho$
  is non-compact, however it can be approximated by a increasing sequence of compact and invariant subsets.
  Indeed, for all $N\ge 1$,  define 
 \begin{quote}
$\Sigma_N:=\{\sul=(s_0s_1\cdots )\in X: \textrm{ for } j\geq 0, s_j
 \in \{-N,\cdots, N\}\},$
\end{quote}   
 so, the following holds

 \begin{lem}\label{Dynamics Lemma}
 \begin{enumerate}
     \item For all $N\ge 1$,   
  $\Sigma_N\subset X$,  $\Sigma_N$ is compact with respect to $\varrho$ and invariant by $\sigma$. Moreover,  for each compact subset $\Lambda$ of $X $ with respect to the metric $\varrho$, so that $\sigma(\Lambda)\subset \Lambda$, we have, there exists $N_0\geq 1$, such that  $\Lambda \subset \Sigma_{N_0}$.

\item 
There exists $\delta_0$ such that the following condition holds
 \begin{quote}
    There exist $C>0$ and $\lambda>1$ such that for every $n\in \N$ and $\sul, \tul\in X$ and $u^*\in \Z^n$, if $\varrho(\sul,\tul)<\delta_0$ then we have 
  $$\varrho(u^*\sul,u^*\tul)\leq C \lambda^{-n}\varrho(\sul,\tul).$$
  \end{quote}
 
\item For every $R>0$ there exists $n\geq 1$ such that  for every $\sul\in B(\zul,R)$, we have 
$
\sigma^n(B(\sul,\delta_0))\supset B(\zul,R).
$ Thus $(X,\sigma) $ is topologically mixing

\item The set
  $\bigcup_{N\geq 1}\Sigma_N $ is dense in $X.$

 \end{enumerate} 
 \end{lem}

 \medskip
 
 For the next theorem, we  remark about the  definition  of  a conformal measure stated in \ref{Conf.Meas.}.
In \S~\ref{Gibbs Measures}, we give the definition of a Gibbs measure of
potentials appropriately adapted
 for the transcendental functions.



\medskip

\begin{teo}\label{teo4}
Let $f\in
\mathscr{F}$.
 Then for every  potential $\phi\in \mathscr{P}_f$ we have the
 following properties.
\begin{enumerate}
\item  There exists  a unique $e^{P(\phi)-\phi}$-conformal measure
$\nu_\phi$ of $f$.
 \item There exists  a unique  probability   Gibbs
state $\mu_\phi$. That is,  $\mu_\phi$ is  $f$-invariant  and
equivalent to $\nu_\phi$. Moreover, both measures are ergodic  and
supported on the radial Julia set $J_r(f)$, where
$$J_r(f)=\{z\in J(f): \lim_{n\rightarrow\infty} f^n(z)< \infty \}.$$
 \item For each $w\in J(f),$ we have
  $ P(\phi)=\displaystyle
 \lim_{n\rightarrow\infty}\frac{1}{n}\log\sum_{z\in f^{-n}(w)}\exp\left(\sum_{j=0}^{n-1}\phi\circ f^{j}(z)\right).$
\end{enumerate}
\end{teo}



\section{Proof of results}

\subsection{Proof of Theorem~1}\label{uno}%
 Consider  the exponential family $\left\{f_\lambda(z)= \lambda e^z, \lambda\in (0,1/e)\right \}$. Each  $f_\lambda$ belongs to $\mathscr{F}$, because  has order equal to $1$,  satisfies the
rapid derivative growth condition 
 with $\alpha_1=0$ and $\alpha_2=1$, and since $0$ is   the only singular value of $f_\lambda$, so this set is hyperbolic.
  Moreover  the potentials $-t\log\vert z\vert= -t\log\vert f'_\lambda(z)\vert+ \log \gamma_1-\log \gamma_1\circ f_\lambda$, where $\gamma_1=\vert z\vert^{-t}$, are tame potentials and   also 
 belong to the class $\mathscr{P}_{f_\lambda}$.

On the other hand, 
let 
 $\D$ be the open unit  disk in $\C$,  then 
 $$f_\lambda\left(\left\{z: \re z< \ln
\left(\frac{1}{\lambda}\right)\right\}\right)=
\D\backslash\{0\}, $$ and since $1<\ln(\frac{1}{\lambda})$ we have 
   $\overline{f_\lambda(\D)}\subset
 \D$.
 Moreover since the immediate basin $B$ of the attracting fixed point   is the  only Fatou  component of $f_\lambda$  we have  $\overline{\D}\subset B$. 
 
Since $f_{\lambda}^{-1}(\C\backslash \overline{\D})=\left\{z: \re z>\ln\left(\frac{1}{\lambda}\right)\right\},$  the  only tract of $f_\lambda$ is the half plane $T=\left\{z: \re
 z>\ln\left(\frac{1}{\lambda}\right)\right\}$.
Let us consider the ray $\alpha:(0,\infty)\rightarrow \C\backslash
\overline{\D}$ defined by
 $\alpha(t)=-(1+t)$, then 
 $$
f_\lambda^{-1}(\alpha(0,\infty))= \bigcup_{k\in\Z}\left\{x+(2k-1)\pi i: x>
\ln\left(\frac{1}{\lambda}\right)\right\}, 
$$ 
and  for each  $k\in\Z$,  put
 $S_k:=\left\{z: \re z>\ln\left(
\frac{1}{\lambda}\right), (2k-1)\pi<\im z<(2k+1) \pi \right\}.$
Then $T\backslash f_\lambda^{-1}(\alpha(t))$  is the disjoint
union of the  fundamental domains   $S_k.$

 Following~\cite{IS06},   let
  $c: J(f_\lambda)\rightarrow \R^+$ be a function such that for
each  $k\in \Z$, is constant  on $J(f_\lambda)~\cap~
(S_{-k}\cup~S_k)$  and we denote by $c_k$ its value on this set. Furthermore we assume that the sequence $(c_k)_{k\in \Z}$ of
positive numbers satisfies
\begin{equation}\label{eq:c_k1}
  \lim_{k\rightarrow \infty}\frac{\log c_k}{\log k}=-\infty.
 \end{equation}
Define
 $\varphi(z):=\log\left(c(z)|z|^{-t}\right),$  where $t>0$, 
 $c(z)=c_k$ if $z\in S_{-k}\cup S_k$ and the sequence
$(c_k)_{k\in \Z}$ satisfies~(\ref{eq:c_k1}).
 Observe that any potential  as above $\varphi(z)=\log (c(z)|z|^{-t})$
satisfies $\lim_{k\rightarrow+\infty}c_k\rightarrow 0,$  so, 
$\varphi$ is not a tame potential however   this potential 
 belongs to the class $\mathscr{P}_{f_\lambda}$ because the 
 function $c$ is bounded on each $S_k$ and 
  $|E'_\lambda(z)|_\theta=|z|$. 


\subsection{Proof Lemma~\ref{Dynamics Lemma}}

 
1. 
By  Denjoy-Carleman-Ahlfors theorem implies 
that the transcendental entire functions of finite order have only a finite numbers of tracts. We will 
 assume for simplicity that for $f\in \mathscr F$ there is only one tract $T$,  and  there is no complication in the generalization to a finite number of tracts.

Let $N\geq 1$ and  denote by $\mathcal{S}_N$ the union of $2N+1$ fundamental domains in $T$, that is 
$$\mathcal{S}_N= \bigcup_{k=-N}^{N}S_k$$
and define
      $$K_N:=\{z\in \mathcal{S}_N: \textrm{ for every }j\geq
 0, f^j(z)\in \mathcal{S}_N\}.$$ 
 
 Let 
   $(\sul^m)_{m\geq 1}$ be a sequence in $\Sigma_N$. Taking  a subsequence if it is necessary, one  can assume that for some $R$ large enough, the subsequence   $(\sul^m)_{m\geq 1}$  is contained in   $\Sigma_N\cap B(\zul,R)$. So for every $m\geq 1$ there is 
   there is  an endpoint $z_{\sul^m}=h_{\sul^m}(0)\in K_N\cap B(z_{\zul},R)$. Since $K_N\cap B(z_{\zul},R)$ is bounded and $K_N$ is closed in $J(f)$  we have there is a subsequence  converging to some point $z\in K_N$. Let $\sul$ be the itinerary associated to $z$, then $\sul\in \Sigma^N_+$ and 
   $\varrho(\sul^{m_j},\sul)=\vert z_{\sul^{m_j}}-z\vert\to 0, j\to\infty.$
   
 
 \vspace{0.5cm}
 
  On the other hand, let
$\Lambda$ be a compact subset of $X$  with respect to the metric~$\varrho$ with $\sigma(\Lambda)\subset\Lambda$. Let  $\sul\in
 \Lambda$ and let $z\in H(\Lambda)$ with itinerary $\sul$, then since   
  the compact subset  $H(\Lambda)$ intersects only a finite
 numbers of tracts (see~\cite[Lemma~$3.2$]{BK07}),
 there exists
 $N_0\geq 1$ such that
  for
 every $j\geq 0$  we have $|s_j|\leq N_0$. Therefore $ \Lambda\subset \Sigma_{N_0}.$

\medskip 

2. 
   Follows from the derivative grown condition and the uniformly expanding property of $f$,
see~\cite[Proposition~4.4]{MU10}.

\medskip



3. Following~\cite[Lemma~4.2]{MU10}, we have  for each  $R>0$ there exists $n\geq 1$ such that  for every $w\in B(z_{\zul},R)$, 
$
f^n(B(w,\delta_0))\supset B(z_{\zul},R),
$ then  the  property follows.
Let $\sul, \tul\in X$ and $U=B(\sul,\delta_1)$,  $V=B(\tul,\delta_2)$, for $\delta_1,\delta_2>0$. 
 By expanding property in Part~2, follows   there is $n_1>0$ such $\sigma^{n_1}(U)\supseteq B(\sigma^{n_1}\sul,\delta_0)$. 
 Let $R$ be large enough such that $\sigma^{n_1}\sul, \tul\in B(\zul,R)$ and there  there is $N_1$ such that 
 $$
 \sigma^{N_1}(B(\sigma^{n_1}\sul,\delta_0))\supset B(\zul,R).
 $$
 So, $
\sigma^{N_1+n_1}(U)\supset 
\sigma^{N_1}B(\sigma^{n_1} \sul,\delta_0)\supset B(\zul,R).
 $
  Taking $m>N_1+n_1$, we have for every $k\geq m$, 
  $
\sigma^{k}(U)\cap V \neq \emptyset.
  $

\medskip


\noindent  4.  Let $\sul=(s_0s_1\cdots) \in X$ and $\epsilon>0$. 
  Then, there exists
  $n_1>0$ such that
  $$ B(\sigma^{n_1}\sul,\delta_0)\subseteq \sigma^{n_1}B(\sul,\epsilon).$$ 
  It is enough to take $n_1>0$ such that $C\lambda^{-n_1}\delta_0<\epsilon.$
  So,  we have
  $$ s_0s_1\cdots s_{n_1-1}B(\sigma^{n_1}\sul,\delta_0)\subseteq
  B(\sul,C\lambda^{-n_1}\delta_0)\subseteq B(\sul,\epsilon).$$
Hence, $ B(\sigma^{n_1}\sul,\delta_0)\subseteq
  \sigma^{n_1}B(\sul,\epsilon).$
Since  $X=\bigcup_{R>0}B(\zul,R)$, then for some  $R>0$ we have
  $\sigma^{n_1}\sul\in B(\zul, R)$.  Moreover  from Part~3,    there is $n_2>0$ such that
 $ B(\zul,R)\subset \sigma^{n_2}B(\sigma^{n_1}\sul,\delta_0).$
  Therefore,
  for
  $n=n_1+n_2$  we  follows that
  $B(\zul,R)\subset \sigma^{n}B(\sul,\epsilon). $ 
  Hence, the set
  $\sigma^n B(\sul,\epsilon)$ contains the sequence
  $\underline{0}=00\cdots$.
  Let  $w^*=w_0\cdots w_{n-1}\in \Z^n$ such that $ w^* B(\zul,R)\subset B(\sul,\epsilon)$,
  then  $w_0\cdots w_{n-1}\zul \in B(\sul,\epsilon)$. Then, taking $N:=\max\{
  |w_0|,\cdots, |w_{n-1}|\}$  we conclude $ w_0 w_1\cdots w_n\zul\in
  \Sigma_N$. See also Lemma~\ref{aul_r} and Lemma \ref{lem: dense}
  
  \medskip

\medskip

Let  $\Cb(J(f),\R)$ be denote the 
Banach space of bounded continuous functions on $J(f)$.
For each potential $\varphi\in\mathscr{P}_f$, the transfer operator associated to $\varphi$  and denoted by $\mathscr{L}_\varphi$ acts 
  continuously  on $\Cb(J(f),\R)$. So for each~$\psi\in \Cb(J(f),\R)$, 
$$\displaystyle \mathscr{L}_{\varphi}\psi(z)
=\sum_{f(w)=z}\psi(w)\exp(\varphi(w)).$$


Following~\cite{MU10}, the thermodynamic formalism was stated for a large class of hyperbolic meromorphic functions
$f:~\C\rightarrow \overline{\C}$ of finite order $\rho'$
satisfying  a rapid growth condition and for a class of
tame potentials.
To prove Theorem~\ref{teo4} one could apply  without  making too many changes the approach  used in ~\cite{MU10}
 to the family of transcendental entire maps  and potentials under study. 
For instance, the  Proposition~\ref{pro: potentials} seems to works without any modifications for potentials $\phi_c$ with $c$ being only bounded from above in~\cite{MU10}.
Otherwise one can obtain
this  result as a consequence of the Theorems~\ref{teo1}, \ref{teo2} and \ref{teo3}, realizing that the class of potentials hyperbolic entire maps and potentials studied fit  in the hypothesis of the results developed in the next section for a symbolic version.


\begin{prop}\label{pro: potentials}
Given $f\in\mathscr{F}$, we have  each potential $\varphi_{c,t}\in \mathscr{P}_f$ satisfies the following properties.
\begin{enumerate}
\item $\displaystyle\sup_{\wul\in X}\mathscr{L}_{\varphi_{c,t}\circ H}(\indicator{})(\wul)<\infty$
\item $\displaystyle\lim_{R\rightarrow \infty}\sum_{s\in\Z}\exp\left(\displaystyle\sup_{\wul\in[s]\cap
(X\backslash B(\zul,R))}\varphi_{c,t}\circ H(\wul)\right)=0.$
\item 
$\displaystyle\lim_{z_{\wul}\rightarrow \infty}\mathscr{L}_{\varphi_{c,t}\circ H}\indicator{}(\wul)=0$
\end{enumerate}

 \end{prop}

\begin{proof}





Let $\varphi_{c,t}= \log c-t\log|f'|_\theta$ be a potential in $\mathscr{P}_f$
 and $\psi\in\Cb(J(f),\R)$, 
 \begin{multline*}
\mathscr{L}_{_{\varphi_{c,t}\circ H}}\psi\circ H(\wul)
=\sum_{\sigma(\sul)=\wul}\psi(z_{\sul})c(z_{\sul})|f'(z_{\sul})|^{-t}_\theta\\
  =\sum_{\sigma(\sul)=\wul}\psi(z_{\sul})c(z_{\sul})|f'(z_{\sul})|^{-t}|z_{\sul}|^{-\tau
 t}|f(z_{\sul})|^{\tau t}\\
= |z_{\wul}|^{\tau t}
 \sum_{\sigma(\sul)=\wul}\psi(z_{\sul})c(z_{\sul})|f'(z_{\sul})|^{-t}|z_{\sul}|^{-\tau
 t}.
 \end{multline*}
Since $f$ satisfies the  derivative growth  condition, we have
\begin{multline*}
 \mathscr{L}_{\varphi_{c,t}\circ H}(\indicator{})(\wul)
  \leq \kappa^t |z_{\wul}|^{\tau t}\sum_{\sigma(\sul)=\wul}c(z_{\sul}) |z_{\sul}|^{-\alpha_1
  t}|f(z_{\sul})|^{-\alpha_2 t}|z_{\sul}|^{-\tau t}\\
= \kappa^t |z_{\wul}|^{\tau t}\sum_{\sigma(\sul)=\wul}c(z_{\sul}) |z_{\sul}|^{-\alpha_1
  t}|z_{\wul}|^{-\alpha_2 t}|z_{\sul}|^{-\tau t}\\
     \leq  \frac{\kappa^{t}}{|z_{\wul}|^{t(\alpha_2-\tau)}}\sum_{\sigma(\sul)=\wul}c(z_{\sul})|z_{\sul}|^{-(\tau+\alpha_1) t}\\\leq
\frac{\kappa^{t}}{|z_{\wul}|^{t(\alpha_2-\tau)}}\sup_{\sul\in\mathcal{S}^\N}
c(z_{\sul})\sum_{\sigma(\sul)=\wul}|z_{\sul}|^{-(\tau+\alpha_1) t}.
  \end{multline*}
Since  $f$  is a transcendental entire function of finite order $\rho$ and $t>\rho/(\tau+\alpha_1)$, then the Borel-Picard~Theorem~(see \cite[Theorem~3.5]{MU10})
 states that the series has  the exponent of  convergence  equal
 to  $\rho$. So the  last sum is finite. Following 
~\cite[Proposition $3.6$]{MU10}, there exists $ \mathcal{M}_t>0$
     such for
 all  $\wul\in X$ we have
 \begin{equation}\label{operator}
\mathscr{L}_{\varphi_{c,t}\circ H}(\indicator{})(\wul)\leq
 \frac{\mathcal{M}_t}{|z_{\wul}|^{t(\alpha_2-\tau)}}\sup_{\sul\in \mathcal{S}^\N}
c(z_{\sul}).
\end{equation}
%
%
So, the equation~(\ref{operator}) implies 
$\displaystyle \lim_{R\rightarrow \infty}\sum_{s\in\Z}\exp\left(\displaystyle\sup_{\wul\in[s]\cap
(X\backslash B(\zul,R))}\varphi_{c,t}\circ H(\wul)\right)=~0$ and 
$\lim_{z_{\wul}\rightarrow \infty}\mathscr{L}_{\varphi_{c,t}\circ H}\indicator{}(\wul)=0.$
\end{proof}

 


 \section{Symbolic dynamic outlook }\label{Appendix}
 
In this section we give a very general approach that will complement the results  in  previous section. To do that, we give  certain conditions to guarantee what is desired  from a symbolic perspective. 

 Let $$\Sigma^+:=\{\aul=(a_0a_1\ldots):a_j\in\Z, \text{ for all } j\geq 0\}=\Z^\N,$$   be the
 full shift space on a countably infinite number of symbols endowed with the product topology, and let
$\sigma:\Sigma^+\rightarrow\Sigma^+$ be the shift transformation
defined by $\sigma(a_0a_1\cdots )=(a_1a_2\cdots).$
 For every  $N\geq 1$ let
 \begin{equation}
\Sigma^+_N:=\{\aul=(a_0a_1\cdots )\in\Sigma^+: a_j
 \in \{-N,\cdots, N\}\}.
\end{equation}

The purpose of this section is to develop a thermodynamic
formalism for the dynamics of  $\sigma$  on  some  subset $X$ of
$\Sigma^+$  which is completely invariant, that is, satisfies
 $$\sigma(X)=X=\sigma^{-1}(X).$$
Thus  the restricted transformation $\sigma|_{X}$ and all of its
 iterates are well defined.
We endow $X$
 with  some metric~$\varrho$  which is not necessarily compatible
 with the product metric on $\Sigma^+$, in the sense that it is not assumed that
the metric $\varrho $ on $X $ can be extended to the full shift in the way that $\sigma$ is still continuous.
 
 We do not assume that $X$ is compact, however,
  we assume that, 
  \begin{quote}
  for  each $N\geq 1$, the set $X$ contains
 $\Sigma^+_N$, and that
 $\Sigma^+_N$ is a compact  with respect to $\varrho$. Moreover we
 assume that  for each  subset $\Lambda\subset X$, that is
 compact  with respect to $\varrho$  such that  $\sigma(\Lambda)\subset\Lambda$ there exists $N\geq 1$
 such that $\Lambda\subset \Sigma^+_N$.
\end{quote}
In Section~\ref{subs1} we give conditions on $(X,\varrho)$ that
implies  existence of  a
 conformal measure and an absolutely continuous invariant
measure  for a certain class  of weakly
 H\"older continuous
potentials, see Theorem~\ref{teo1}, Theorem~\ref{teo2} and
Theorem~\ref{teo3}. 

\subsection{Conditions on the space $(X,\varrho)$ and the main
results}\label{subs1}
 Given $\aul=(a_0 a_1\ldots)\in \Sigma^+$  and given a word of finite length
 $b^*=b_0\cdots b_{n-1}\in~\Z^n$ let us write
 $$
b^*\aul=(b_0\cdots b_{n-1}a_0 a_1\cdots).
 $$
 Given a set $A\subset \Sigma^+$ let us write $b^*A=\left\{b^*\aul: \aul\in
 A\right\}$. For $\aul\in X$ and $\delta>0 $ we define
$$ B(\aul,\delta):=\big \{\bul\in X:\varrho(\aul,\bul)<\delta\big \},$$
$$\B_0(\aul,\delta):=\big \{\bul\in X:\varrho(\aul,\bul)<\delta \textrm{ and }
   b_0=a_0\big \}.$$
   Let us denote
   $$\overline{B(\zul,R)}=\left\{\bul\in X: \varrho(\zul,\bul)\leq R \right\}.$$
By $\B_n(\aul,\delta)$ we denote  the $(n,\delta)$-ball at
$\aul\in X$, namely
$$
\B_n(\aul,\delta):=\left \{\bul\in X:\sigma^j(\bul)\in
 \B_0(\sigma^j(\aul),\delta), \textrm{ for  all }
j=0,1,\cdots,n\,
 \right\}.$$

 We assume that there exists $\delta_0>0$
such that the metric $\varrho$ satisfies the \emph{exponential
shrinking of  preimages on balls}
 $(\ExpShr)$ and \emph{connection on
balls} $(\Conn)$ given below.
\begin{itemize}
\item[$(\ExpShr)$] There exist  $C_{\ES}>0$ and $\lambda_{\ES}>1$
such that, for all $n\in\N$, $\aul,\bul\in X$ and $c^*\in\Z^n$,
 if $\varrho(\aul,\bul)<\delta_0$ then we have
\begin{displaymath}
\varrho(c^*\aul,c^*\bul)\leq
C_{\ES}\lambda_{\ES}^{-n}\varrho(\aul,\bul).
\end{displaymath}
\end{itemize}
 \begin{itemize}
\item[$(\Conn)$] For every $R>0$ there exists $n\geq 1$
 such that  for all $\aul\in B(\zul,R)$
 we have  $$\sigma^n B(\aul,\delta_0)\supset
B(\zul,R).$$
\end{itemize}
  So the plan  here is as follows.
\subsubsection*{Weakly H\"older continuous potentials}
  Given $\delta_1>0$  and $\alpha\in~(0,1]$. We say that  a function
      $\phi:X\rightarrow \R$ is
  \emph{uniformly} $\delta_1$-\emph{locally}
  $\alpha$-\emph{H\"older} 
  if there exists $L\geq 0$ such
  that  for all $\aul,\bul,\cul\in X$ satisfying $\aul,\bul\in\B_0(\cul,\delta)$
     we have
  $$|\phi(\aul)-\phi(\bul)|\leq L(\varrho(\aul,\bul))^\alpha.$$
 Given $\delta\in(0,\delta_0]$,
 $n\geq 0$ we define the
 $(n,\delta)$-\emph{variation} of  the potential
 $\phi:X\rightarrow\mathbb{R}$ (See~\cite{Sa99}) by
$$
 \Var_n(\phi):=\sup_{\aul\in X}\sup_{\bul,\cul\in
\B_n(\aul,\delta)}|\phi(\bul)-\phi(\cul)|.
$$
Given $r\in(0,1)$ and $\delta\in (0,\delta_0]$, we say that a
potential $\phi:X\rightarrow \R$ is $(\delta,r)$-\emph{weakly
H\"older continuous} if there exist $C>0$
 such that  for every $n\geq 0$
$$\Var_n(\phi)\leq C r^n.$$
We denote by $\BV_r(X,\mathbb{R})$   the space of all
$(\delta,r)$-weakly H\"older continuous. We work with  fixed
$\alpha\in(0,1)$ and $ \delta\in (0,\delta_0]$ and we prove that
every uniformly $\delta$-locally $\alpha$-H\"older continuous
potential is $(\delta,r)$- weakly H\"older continuous with
$r=\lambda_{\ES}^{-\alpha} $ (Lemma~\ref{lem:delta1}). We assume
further that the following  condition $(\Delt)$  is satisfied.
 \begin{itemize}
\item[$(\Delt)$] If $\phi$ is $(\delta',r)$-\emph{weakly H\"older
continuous}   then
 $\phi$
 is  $(\delta,r)$-\emph{weakly H\"older continuous}, where  $\delta'=\min\{\delta,\delta/C_{\ES}\}.$
\end{itemize}
 Since this  condition could be a bit difficult to verify directly,
 we show that the topological condition ($(\Path)$)  implies
$(\Delt)$ (Proposition~\ref{weakly}).
\begin{itemize}
\item[$(\Path)$] There exists $\ell\geq 1$ such that for all $\aul\in
X$ and for all $ \bul,\cul\in\B_0(\aul,\delta)$, there exists a
sequence
$$\bul=\aul_0,\aul_1,\cdots,\aul_\ell=\cul,$$
such that for all $ j\geq 0$ we have
 $\varrho(\aul_j,\aul_{j+1})<\delta'$.
\end{itemize}
\subsubsection*{The transfer operator and the pressure}
We say that  a potential $\phi$ is \emph{summable} if it satisfies
  \begin{equation}\label{sumable}
  \sup_{\aul\in X}\left\{\sum_{\bul: \sigma(\bul)=\aul}\exp(\phi(\bul))
  \right\}<\infty.
  \end{equation}
We say that $\phi$ is \emph{bounded on balls} if
  for all $R>0$ we have
  \begin{equation}\label{balls}
  \sup_{\cul\in B(0,R)}|\phi(\cul)|<\infty.
  \end{equation}
If we assume  that the potential $\phi$ satisfies (\ref{sumable})
and (\ref{balls}) then for any  real bounded  continuous  function
$g$ on $X$,
  the \emph{transfer operator}
  $$  \mathscr{L}_\phi(g)(\aul):=\sum_{\bul:\sigma(\bul)=\aul}e^{\phi(\bul)}g(\bul),
$$
 is  well-defined.
 We denote the space of real bounded continuous functions on
$X$  by  $\Cb_b(X,\R)$. For every $n\geq 1$ and $\aul\in X$ we put
$$S_n \phi (\aul)=\sum_{k=1}^{n-1}\phi\circ\sigma^k(\aul).$$
 So, for every $n\geq 1$ and $\aul\in X$ we have
 $$\mathscr{L}^n_\phi(g)(\aul)=\sum_{\bul: \sigma^n(\bul)=\aul}e^{S_n \phi (\bul)}g(\bul).$$

 Since our space $X$ is not compact the topological pressure is
 not defined in the standard way, so we define the
 \emph{pressure} $P(\phi)$ of a potential weakly H\"older
 continuous  $\phi$ with respect to $\sigma$ as the supremum over
 all the topological pressures of $\phi$ restrict to $\Sigma_N^+,$
 that is
\begin{equation}\label{topol.press}
P(\phi):=\sup_{N\geq 1} P_{\Top}(\phi|_{\Sigma^+_N}),
\end{equation}
where
  $P_{\Top}(\phi|_{\Sigma^+_N})$ is the  topological pressure of
$\phi$ over  $\Sigma^+_N$, that is defined  
for
 each $N\geq 1$
 $$P_{\Top}(\phi|_{\Sigma^+_N}):=\lim_{n\rightarrow\infty}
\frac{1}{n}\log \sum_{c^*\in \Z^n, c^*\aul\in
\Sigma^+_N}e^{S_n \phi (c^*\aul)}.$$

\medskip

 To guarantee the existence of  a
conformal measure we  have to impose certain conditions on $X$,
$\rho$, and the potential $\phi$. Namely, notice that  due to the
lack of compactness of $X$ the Schauder-Tychonoff Fixed Point
Theorem cannot be applied. 

\medskip

  For any $a^*\in\Z^n$ the cylinder $[a^*]$
 is defined by
$$[a^*]:=\left\{\bul\in X: b_i=a_i, 0\leq i\leq n-1\right\}.$$
Let us assume that  the following conditions are satisfied.
\begin{enumerate}
\item The transformation
 $\sigma:X\to
X$ can be continuously extended to
$\overline{\sigma}:\X^\varrho\to \X^\varrho$, where $\X^{\varrho}$
is the completion of $X$ with respect to the metric $\varrho$.
\item
 For any $R>0$,  the set
  $\overline{B(\underline{0},R)}$ is compact in $\overline{X}^\varrho$.
\item For every $R>0$ and for every $\bul\in
\overline{X}^\varrho\sms X$ there exists $N\geq 1$ such that for
$n\geq N$, we have $\overline{\sigma}^n(\bul)\in
\overline{X}^\varrho\sms \overline{B(\underline{0},R)}$.
 \item
Given $R>0$ and $k\in \Z$, let
$$[k,R]:=\{\aul\in [k]\cap X:\varrho(\underline{0},\aul)>~R\}.$$
Assume that every  potential $\phi\in \BV_r(X,\mathbb{R})$,
 satisfies
\begin{equation}\label{eq:Cond.Potential1}
\sum\limits_{k\in \Z}e^{\sup \phi |_{[k ]}}<\infty,
\end{equation}
\begin{equation}\label{eq:Cond.Potential2}
\lim_{R\rightarrow \infty}\sum\limits_{k\in\Z}e^{\sup\phi
|_{[k,R]}}=0.
\end{equation}
\item Every potential $\phi\in \BV_r(X,\R)$ has a continuous
extension to $\overline{X}^\varrho$.
\end{enumerate}
If it does not lead to misunderstanding we will frequently  denote
$\overline{\sigma}$ simply by $\sigma$.

 Set
   $$X_{\rad}:=\left\{\aul\in \overline{X}^\varrho: \omega_{\sigma}(\aul)\cap X\neq
 \emptyset\right\}.$$
 Let us  state the main results of this paper.
 The proofs can be found in  Section~\ref{Conf.Meas.}.
\begin{teo}\label{teo1} Let $\sigma: X\rightarrow X$ be the shift map and let  $\phi\in \BV_r(X,\R)$ satisfying the
conditions  $(\ie)$ $(\iie)$, $(\iiie),$ $(\ive)$ and $(\ve)$.
Then
 there exists a measure $\nu$ that is
 $e^{P(\phi)-\phi}$-conformal and satisfies
 $\nu(X_{\rad})=1.$
\end{teo}
We  obtain  the measure $\nu$ 
 as  the
weak limit of a \emph{tight} sequence of measures
  $\{\nu_{_N}\}_{N\in \mathds{N}}$,
 where every measure  $\nu_{_N}$  is
 conformal with respect to  $\sigma|_{\Sigma^+_N}.$ We recall  the definition  of   \emph{tightness}  in Section~\ref{Conf.Meas.}.
\begin{teo}\label{teo2}
Let $\sigma: X\rightarrow X$ be the shift map  and let $\phi\in
\BV_r(X,\R)$ satisfying the  conditions $(\ie)$ $(\iie)$,
$(\iiie),$ $(\ive)$ and $(\ve)$. Then
  there exists an invariant measure $\mu$
 that is
absolutely continuous with respect to
  $\nu$ and is a Gibbs measure  for $\phi.$
\end{teo}
\noindent  Moreover
 we say that  a potential $\phi\in \BV_r(X,\mathbb{R})$ is  \emph{rapidly decreasing}
 if
 \begin{equation}\label{decreasing}
\lim_{R\rightarrow\infty}\sup_{\aul\in
\overline{X}^\varrho\backslash B(\zul,R)}
\mathscr{L}_\phi\mathds{1}(\aul)=0.
\end{equation}
Under this additional assumption we are able to prove the
following result.
\begin{teo}\label{teo3}
 Let  $\phi$ be a potential  satisfying $(\ref{sumable})$, $(\ref{balls})$ and
 $(\ref{decreasing})$.
 Then  for each $\aul \in
\overline{X}^\varrho$ we have that  the  limit
$\lim_{n\rightarrow\infty}\frac{1}{n}\log
\mathscr{L}^n_{\phi}\mathds{1}(\aul)$ exists and is equal to the
pressure $P(\phi)$. This limit is independent of $\aul.$
\end{teo}

\medskip

 Notice that  from~(\ref{operator}) each potential $\phi_{c,t}\in \mathscr{P}_f$
satisfies the the condition of summability   given
in~(\ref{sumable}).  The condition of bounded on ball given in~(\ref{balls}) follows from the rapid growth  condition.
 Moreover, from equation~(\ref{operator}) implies that the
condition (\ref{eq:Cond.Potential1}) is satisfied.
 The condition (\ref{eq:Cond.Potential2})  follows  since  for  $R>0$ and   $U_R:=\left\{z\in \C: |z|>R\right\}$, 
we have that for every  $\phi_{c,t}\in
\mathscr{P}$
and   for every $w\in J(f)$,
$$
\lim_{R\rightarrow \infty}\sum_{S\in
\mathcal{S}}\exp\left(\displaystyle\sup_{z\in S\cap
U_R}\phi_{c,t}(z)\right)=0.
$$
 and 
 $$\lim_{w\rightarrow \infty}\mathscr{L}_{\phi_{c,t}}\mathds{1}(w)=0.$$ 
 Thus  Theorem~\ref{teo4} can  follow as a consequence of the Theorem~\ref{teo1}, Theorem~\ref{teo2} and Theorem~\ref{teo3}.

\subsection{Some preliminary Lemmas}

Throughout all this section we let $X$ be a subset of $\Sigma^+$
such that $\sigma(X)=X$ and $\sigma^{-1}(X)=X$. Furthermore
we assume $X$ is  endowed with a metric$~\varrho$ satisfying
$(\ExpShr)$ with constants $\delta_0>0$, $C_{\ES}>0$ and
$\lambda_{\ES}> 1$.

 The following properties follow immediately from the definitions,
and will be used several times. For all $\delta>0$, $\aul\in X$,
$n,m\geq 0$  and $c^*\in \Z^n$ we have
\begin{equation}\label{pordefi1}
 \sigma^n(\B_{m+n}(\aul,\delta))\subseteq
 \B_{m}(\sigma^n\aul,\delta),
 \end{equation}
\begin{equation}\label{pordefi2}
\textrm{ and } \B_{m+n}(c^*\aul,\delta)\subseteq
c^*\B_{m}(\aul,\delta).
\end{equation}
Then the following lemmas hold.
\begin{lem}\label{TresLemas}
Fix $n_0\geq 0$ such that
$C_{\ES}\lambda_{\ES}^{-n_0}\leq\min\{1,1/C_{\ES}\}.$ Then for all
$\delta\in (0,\delta_0]$,\, $n, m\geq 0$, $c^*\in \Z^n$ and
$\aul\in X$, we
  have
\begin{enumerate}
\item[1.] \label{lem:ESc}$\B_{m+n}(c^*\aul,\delta)\subseteq
c^*\B_m(\aul,\delta)\subseteq
\B_m(c^*\aul,C_{\ES}\lambda_{\ES}^{-n}\delta). $\\
 In particular taking
$n=n_0$, we have
$$\B_{m+n_0}(c^*\aul,\delta)\subseteq\B_{m}(c^*\aul,\min\{\delta,\delta/C_{\ES}\}).$$
\item[2.] \label{lem:delta2}
$\B_{m+n}(c^*\aul,\min\{\delta,\delta/C_{\ES}\})\subseteq
c^*\B_m(\aul,\min\{\delta,\delta/C_{\ES}\})\subseteq\B_{m+n}(c^*\aul,\delta).$
  \end{enumerate}
\end{lem}
\begin{proof}
$\mathbf{1.}$  The  first inclusion is (\ref{pordefi2}).  The
second inclusion  follows since for $\bul\in
c^*\B_{m}(\aul,\delta)$ there is $\bul'\in \B_{m}(\aul,\delta)$
such that $\bul=c^*\bul'$ and for~ each $j=0,\cdots,m $ we have
 $$\varrho(\sigma^j\aul,\sigma^j\bul')<\delta \textrm{ and }
 \aul_j=\bul_j.
 $$
  Then by $(\ExpShr)$ implies that for such $j$ we have
 $$\varrho(\sigma^j(c^*\aul),\sigma^j(c^*\bul'))<C_{\ES}\lambda_{\ES}^{-n}\delta.$$
 \noindent $\mathbf{2.}$ Set  $\delta'=\min\{\delta,\delta/C_{\ES}\}$.
  The first inclusion is (\ref{pordefi2}) with
 $\delta$ replaced by  $\delta'$. To prove the
 second inclusion let $\bul\in
 c^*\B_{m}(\aul,\delta')$. Then there exists
 $\bul'\in \B_{m}(\aul,\delta')$ such that
 $b=c^*\bul'$, hence for all $j=0,1,\cdots,m $ we have
  $$\varrho(\sigma^j\bul',\sigma^j\aul)<\delta' \textrm{ and }
  \bul'_j=\aul_j.$$ Moreover by $(\ExpShr)$ for each $j=0,1,\cdots,n$ we
  have
 $$\varrho(\sigma^j(c^*\bul'),\sigma^j(c^*\aul))\leq C_{\ES}
\lambda_{\ES}^{j-n}\varrho(\bul',\aul)<C_{\ES}\delta'\leq
\delta.$$
\end{proof}
\begin{lem}\label{aul_r}
For all  $\aul\in X$ and $r>0$ there exists  $n>0$ such that
$$ B(\sigma^{n}\aul,\delta_0)\subseteq \sigma^{n}B(\aul,r).$$
\end{lem}
\begin{proof}
  Take $n>0$ such that $C_{\ES}\lambda_{\ES}^{-n}\delta_0<r.$
Thus  by $(\ExpShr)$ we have
$$ a_0a_1\cdots a_{n-1}B(\sigma^{n}\aul,\delta_0)\subseteq
B(\aul,C_{\ES}\lambda_{\ES}^{-n}\delta_0)\subseteq B(\aul,r).$$
And so we have
$$ B(\sigma^{n}\aul,\delta_0)\subseteq
\sigma^{n}B(\aul,r).$$
\end{proof}
\begin{lem}\label{lem: dense}
If  the metric  $\varrho$ satisfies $(\Conn)$, then  the set
$\bigcup_{N\geq 1}\Sigma^+_N $ is dense in $X.$
\end{lem}

\begin{proof}
 Let $\aul\in X$ and $r>0$. 
 Then
by  Lemma~\ref{aul_r} there exists $n_1>0$ such that
\begin{equation}\label{eq1a}
B(\sigma^{n_1}\aul,\delta_0)\subset \sigma^{n_1}B(\aul,r).
\end{equation}
Since  $X=\bigcup_{R>0}B(\zul,R)$, then for some  $R>0$ we have
$\sigma^{n_1}\aul\in B(\zul, R)$.  Moreover by $(\Conn)$ there
exists $n_2>0$ such that
\begin{equation}\label{eq1b}
B(\zul,R)\subset \sigma^{n_2}B(\sigma^{n_1}\aul,\delta_0).
\end{equation}
Therefore from (\ref{eq1a}) and (\ref{eq1b}) we have for
$n=n_1+n_2$  it follows
$B(\zul,R)\subset \sigma^{n}B(\aul,r). $ 
 Hence
  $\sigma^n B(\aul,r)$ contains the sequence
 $\underline{0}=00\cdots$.
Let  $c^*\in \Z^n$ such that $ c^* B(\zul,R)\subset B(\aul,r)$,
then  $c_0\cdots c_{n-1}\zul \in B(\aul,r)$. So, taking $N:=\max\{
|c_0|,\cdots, |c_{n-1}|\}$  we conclude $ c_0c_1\cdots c_n\zul\in
\Sigma^+_N$.
\end{proof}
 For all $\aul,\bul\in X$ and $k\geq 1$, we define
 $$\varrho_k(\aul,\bul):=\max_{0\leq j\leq k
 }\left\{\varrho\left(\sigma^j\aul,\sigma^j\bul\right)\right\}.$$
\begin{lem}\label{lem:path}
Let $\delta\in (0,\delta_0]$ and
$\delta'=\min\{\delta,\delta/C_{\ES}\}$. Assume that $\varrho$
satisfies $(\Path)$.
 Then for all
 $k\geq 1$, $\dul\in X$ and
$\aul,\bul\in \B_{k}(\dul,\delta)$ we have that  there exists a
sequence $\aul=\cul_0,\cul_1,\cdots,\cul_{\ell^k}=\bul$ such that
for every $j\in\{0,\cdots,\ell^k-1\}$,
$$\varrho_k(\cul_j,\cul_{j+1})<\delta'.$$
\end{lem}
\begin{proof}
We will proceed  by induction  in $k$. By assumption $(\Path)$,
the desired assertion is satisfied  for  $k=1$.
 Suppose that this is true for $k>1$. For $\aul,\bul\in
\B_{k+1}(\dul,\delta)$
 we have $\sigma(\aul),\sigma(\bul)\in  \B_k(\sigma(\dul),\delta)$, thus
 there exists a sequence $\sigma(\aul)=\cul_0,\cdots,\cul_{\ell^k}=\sigma(\bul)$
 such that for every $j\in\{0,\cdots,\ell^k-1\}$ we have
$$\varrho_k(\cul_j,\cul_{j+1})< \delta'.$$
Then, let $c^*$ the first word from $\dul$ such that by item $1$
of Lemma~\ref{TresLemas}
$c^*\B_0(\sigma(\dul),\delta')\subset\B_0(c^*\,\sigma(\dul),\delta)$
and for all $j=0,1,\cdots, \ell^k-1$,
$$\varrho(c^*\cul_{j},c^*\cul_{j+1})<\delta.$$
Then for every $j=0,\cdots,\ell^k-1$\,  there exists a sequence
$$c^*\cul_j=\dul^j_0,\dul^j_1,\cdots,\dul^j_\ell=c^*\cul_{j+1},$$ such
that for every
  $0\leq r\leq
\ell-1$ we have
 $\varrho(\dul^j_r,\dul^j_{r+1})<\delta'$.
So there is a path
$$\aul=\dul^0_0,\cdots, \dul^{0}_{\ell}, \dul^1_0,\cdots,d^1_\ell,\cdots,\dul^{\ell^k-1}_1, \cdots, \dul^{\ell^k-1}_\ell=\bul, $$
such that for all $0\leq j\leq \ell^k-1$  and $ 0\leq r\leq
\ell-1$ we have
$$\varrho_{k+1}(\dul^{j}_{r},\dul^{j}_{r+1})<\delta' .$$
\end{proof}

\subsection{Potentials}\label{pot}
  For the rest of this paper
we fix  $\alpha\in (0,1]$ and $\delta\in (0,\delta_0]$. Given
$\delta_1>0$.
 We say that
a function $\phi\in C(X,\R)$ is \emph{uniformly $\delta_1$-locally
$\alpha$-H\"{o}lder} $(\ULH)$ if,
 there exists $L\geq 0$ such that.
\begin{quote}
For all $\aul,\bul, \cul\in X$, satisfying $\aul,\bul\in
\B_0(\cul,\delta)$
   we have
   \begin{displaymath}
    \qquad \qquad\qquad     |\phi(\underline{a})-\phi(\underline{b})|\leq
   L(\varrho(\underline{a},\underline{b}))^\alpha.
  \end{displaymath}
  \end{quote}
We say that  a function $\phi$ is \emph{locally $\alpha$-H\"older}
if there is $\delta_1>0$ such that  $\phi$ is $\delta_1$-locally
$\alpha$-H\"older.

We denote by  $H_{\alpha,\delta}$ the vector
 space of all bounded uniformly $\delta$-locally $\alpha$-H\"{o}lder
potentials. We endow  $H_{\alpha,\delta}$ with the norm
$\|\cdot\|_{\alpha,\delta}$ defined as follows.
 First, we define the $\alpha,\delta$-\emph{variation} of $\phi$ by
 \begin{eqnarray*}
  v_{\alpha,\delta}(\phi):=
  \inf\left\{L\geq 0: \textrm{ for all } \aul,\bul\in X, \right.\\
\qquad\qquad \qquad \left.  \textrm{ if }
   \varrho(\aul,\bul)<\delta\textrm{ then }|\phi(\aul)-\phi(\bul)|\leq L\varrho(\aul,\bul)^{\alpha}\right\},
\end{eqnarray*}
and $\|\phi\|_\infty=\sup\left\{|\phi(\aul)|: \aul\in X\right\}$.
Then the norm $\|\cdot\|_{\alpha,\delta}$ on $H_{\alpha,\delta}$
is defined by
\begin{displaymath}
  ||\phi||_{\alpha,\delta}:=v_{\alpha,\delta}(\phi)+||\phi||_{\infty}.
\end{displaymath}
The vector space $H_{\alpha,\delta}$ endowed with the norm
$||\cdot||_{\alpha,\delta}$ is a Banach space.

 For every potential $\phi\in C(X,\R)$ and for every integer
$n\geq 0$  we recall that the  $n$-\emph{variation} of $\phi$ is
defined by
\begin{equation}
\label{eq:Vnphi}
  \Var_n(\phi):=\sup_{\underline{a}\in X}\,
  \sup_{\underline{b},\underline{c}\in \B_n(\underline{a},\delta)}
  |\phi(\underline{b})-\phi(\underline{c})|.
\end{equation}
We say  $\phi\in C(X,\mathbb{R})$  is  \emph{weakly H\"{o}lder
continuous} if it satisfies the following condition.
\begin{quote}
 There exist $C>0$ and $0<r<1$ such that, for all
$n\geq 0$ we have
  \begin{displaymath}
\qquad\qquad \quad    \Var_n(\phi)\leq C r^n.
  \end{displaymath}
\end{quote}
We denote by $\BV(X,\mathbb{R})$  the space of all  weakly
H\"{o}lder continuous potentials. We put
\begin{eqnarray*}
\BV_r(X,\mathbb{R}):= \left\{\phi\in C(X,\mathbb{R}):
\Var_n(\phi)\leq
Cr^n, \right.\\
\qquad\quad\qquad\qquad\qquad\left. \textrm{ for all } n\geq
0,\textrm{ and for some }C\geq 0\right\},
 \end{eqnarray*}
and for $\phi\in \BV_r(X,\mathbb{R})$, let
$$w_r(\phi):=\sup\left\{\displaystyle\frac{\Var_n(\phi)}{r^n}: n\geq
1\right\},$$ and $$\|\phi\|_r=\|\phi\|_\infty+ w_r(\phi).$$
\begin{prop}
The function $\|\cdot\|_r$ is a norm on $\BV_r(X,\mathbb{R})$ that
 makes  it a  Banach space.
\end{prop}
\begin{proof}
 Follows easily that $\|.\|_r$ is a norm in $BV(X,\mathbb{R})$.
To prove that  $BV(X,\mathbb{R})$ is a Banach space with
 the norm $\|.\|_r$, let $\{\phi_{_{(n)}}\}_{n\geq 1}$ be a  Cauchy
 sequence in $BV(X,\mathbb{R})$. So for every $\epsilon >0$,
 there exists $N\in \mathbb{N}$ such that for all $m,n\geq N$
 $\|\phi_{_{(n)}}-\phi_{_{(m)}}\|_\infty+w_r(\phi_{_{(n)}}-\phi_{_{(m)}})_r<\epsilon$.
  Since $\|.\|_\infty\leq \|.\|_r$ and $w_r(.)\leq\|.\|_r$, then
  $\{\phi_{_{(n)}}\}_{n\geq 1}$ is a Cauchy sequence with respect to the
  norm   $\|.\|_\infty$ (this implies  there exists  $\phi\in C(X, \mathbb{R})$ such that $\phi_{_{(n)}}\rightarrow \phi$, $n\rightarrow\infty$)
   and  the sequence $\{w_r(\phi_{_{(n)}})\}_{n\geq
  1}$ is also a Cauchy sequence.

  Observe first that $\phi\in BV(X,\mathbb{R})$.  Let $\epsilon>0$
  then there exists $N$ such that  for all $m\geq N$
  $|\phi_{_{(m)}}(\bul)-\phi(\bul)|<\epsilon$, for all $\bul\in
  X.$ Then for all $\bul,\cul\in \B_n(\aul,\delta)$ with $\aul\in X$
  we have
  \begin{multline*}
|\phi(\bul)-\phi(\cul)|\leq
|\phi(\bul)-\phi_{_{(m)}}(\bul)|+|\phi_{_{(m)}}(\bul)-\phi_{_{(m)}}(\cul)|+
|\phi_{_{(m)}}(\cul)-\phi(\cul)|\\
\leq \frac{\epsilon}{2}+
C_{_{(m)v}}r^n_{_{(m)v}}+\frac{\epsilon}{2}=C_{_{(m)v}}r^n_{_{(m)v}}+\epsilon
  \end{multline*}

  Let $C_{_{(N)r}}=\max\{C_{_{(m)r}}: m\geq N\}$ and $r_{_{(N)r}}=\max\{r_{_{(m)r}}: m\geq
  N\}$, then in particular we have
$$|\phi(\bul)-\phi(\cul)|\leq C_{_{(N)v}}r^n_{_{(N)v}}+\epsilon.$$
Therefore
$$V_n(\phi)\leq C_{_{(N)v}}r^n_{_{(N)v}}, \text{ for all  } n\geq 1.$$
\end{proof}

\begin{lem}\label{lem:delta1}
 Every uniformly $\delta$-locally $\alpha$-H\"older
continuous potential is weakly H\"older continuous with constant
$r=\lambda_{\ES}^{-\alpha}$.
\end{lem}
\begin{proof}
By part $1$ of Lemma~\ref{TresLemas} with $c^*\aul$ replaced by
$\aul$ and with  $m=0$, we obtain
$$\B_{n}(\aul,\delta)\subseteq\B_0(\aul,C_{\ES}\lambda_{\ES}^{-n}\delta).$$
So, if $\phi$ is uniformly $\delta$-locally $\alpha$-H\"older,
then we have that
 we have that there exists $L\geq 0$ such that for all $n\geq
0$, $\aul\in X$ and
  $\bul,\cul \in \B_n(\underline{a},\delta)$ we have
  \begin{displaymath}
    |\phi(\bul)-\phi(\cul)|\leq L\varrho(\bul,\cul)^{\alpha}\leq
    L(C_{\ES}\lambda_{\ES}^{-n})^{\alpha}\delta^\alpha=LC_{\ES}^{\alpha}\delta^\alpha
    (\lambda_{\ES}^{-\alpha})^{n}.
  \end{displaymath}
\end{proof}

 For every $n\geq 1$ put
 $$S_n \phi =\sum_{k=0}^{n-1}\phi\circ \sigma^{k}.$$
If $n=0$, then put  $\phi_0\equiv 0$.
\begin{lem}\label{lem:varrho}
Let $\phi\in \BV_r(X,\mathbb{R})$, and assume that $(\Delt)$
holds. Then for all $n\geq 0$
 we have the following
 assertions.
  \begin{itemize}
  \item[1.] For all $m\geq 0$
  we have
\begin{displaymath}
\qquad\quad V_{m+n}(S_n \phi )\leq  \left(
w_r(\phi)\frac{r^n}{1-r}\right)r^m.
\end{displaymath}
  \item[2.] For every $c^*\in \Z^n$ and
   for all $\aul\in X$, let
  \begin{displaymath}
\qquad\qquad\qquad \phi_{n,c^*}(\aul):=S_n \phi (c^*\aul),
 \end{displaymath}
then $\phi_{n,c^*}\in \BV_r(X,\mathbb{R})$.
  \end{itemize}
\end{lem}
\begin{proof}
$\mathbf{1.}$ Let  $\dul\in X$ and $m\geq 0$. By property
(\ref{pordefi1}) we have that for each  $j=0,1,\cdots,n-1$,
$$\sigma^j\B_{m+n}(\dul,\delta)\subset \B_{m+n-j}(\sigma^j\dul,\delta).$$
  Then, for each $\aul,\bul\in\B_{m+n}(\dul,\delta)$ we have
  \begin{multline}\label{phi_nc}
\left|S_n \phi (\aul)-S_n \phi (\bul)\right|\leq\sum_{j=0}^{n-1}\left|\phi(\sigma^j\aul)-\phi(\sigma^j\bul)\right|
  \\  \leq \sum_{j=0}^{n-1}V_{m+n-j}(\phi)
 \leq
    w_r(\phi)\sum_{j=0}^{n-1} r^{(m+n-j)}
   \\  \leq w_r(\phi)r^{m+n}\frac{1}{1-r}
   = \left( w_r(\phi)\frac{r^n}{1-r}\right)r^m.
  \end{multline}
$\mathbf{2.}$  Let $n\geq 0$,\, $c^*\in \Z^n$ and $m\geq 0$.
\noindent  We have to prove that $\phi_{n,c^*}$ is
$(\delta,r)$-weakly H\"older, however by condition $(\Delt)$  is
enough to show that $\phi_{n,c^*}$ is $(\delta',r)$-weakly
H\"older. In fact, let $\dul\in X$ and  $\aul,\bul\in
\B(\dul,\delta')$. By part $2$ of Lemma~\ref{TresLemas} we have
$$c^*\B_{m}(\dul,\delta')\subseteq\B_{m+n}(c^*\dul,\delta).$$
Hence, using the estimation (\ref{phi_nc}) of part 1 of this lemma
we get, for $c^*\aul,c^*\bul\in \B_{m+n}(c^*\dul,\delta)$
$$\left|S_n \phi (c^*\aul)-S_n \phi (c^*\aul)\right|\leq \left(w_r(\phi)\frac{r^n}{1-r}\right)r^m.$$
It is the same  to have
$$\left|\phi_{n,c^*}(\aul)-\phi_{n,c^*}(\aul)\right|\leq \left(w_r(\phi)\frac{r^n}{1-r}\right)r^m.$$
Thus it follows that   the function $\phi_{n,c^*}$ is
$(\delta,r)$-weakly H\"older continuous.
\end{proof}
\begin{lem}\label{lem:eq5}
Let $\phi\in \BV_r(X,\mathbb{R})$.
 Then
for all $n\geq 1$, $c^*\in \Z^n$, $m\geq
 1$, $\dul\in X$ and $\aul,\bul\in
\B_{m}(\dul,\delta)$, if  we put $K:=w_r(\phi_{n,c^*})$, then
  \begin{equation*}
\qquad\qquad\qquad
\left|e^{S_n \phi (c^*\aul)}-e^{S_n \phi (c^*\bul)}\right|\leq
e^{K}\left(e^{S_n \phi (c^*\bul)}\right ).
  \end{equation*}
\end{lem}
\begin{proof}
Let $m\geq 1$, $\dul\in X$ and \,  $\aul,\bul\in
\B_{m}(\dul,\delta)$. By item 2 of Lemma~\ref{lem:varrho}, we have
for $n\geq 1$ and $c^*\in \Z^n$ the function $\phi_{n,c^*}\in
\BV_r(X,\mathbb{R})$. Thus
\begin{displaymath}
  |S_n \phi (c^*\aul)-S_n \phi (c^*\bul)|\leq w_r(\phi_{n,c^*})r^m\leq w_r(\phi_{n,c^*})=K.
\end{displaymath}
Then
\begin{displaymath}
  e^{S_n \phi (c^*\aul)-S_n \phi (c^*\bul)}-1\leq e^{K}-1,
\end{displaymath}
and  \begin{displaymath}
  1-e^{-(S_n \phi (c^*\bul)-S_n \phi (c^*\aul))}\leq
  1-e^{-K} \leq e^{K}-1.
\end{displaymath}
Therefore $$|e^{S_n \phi (c^*\aul)-S_n \phi (c^*\bul)}-1|\leq e^K-1.$$
\end{proof}
 In the following proposition we prove that  the  condition $(\Delt)$ follows   immediately as consequence of the
assumption $(\Path)$.
\begin{prop}\label{weakly}
Let $\delta'=\min\{\delta,\delta/C_{\ES}\}$. If  the metric
$\varrho$ satisfies the condition $(\Path)$ then   $\phi$
satisfies the condition $(\Delt)$,  it means   if $\phi$ is
$(\delta',r)$-weakly H\"older then $\phi$ is $(\delta,r)$-weakly
H\"older.
\end{prop}
\begin{proof}
Let $\phi$ be a $(\delta',r)$-weakly H\"older potential. Let
$\dul\in X$, $m\geq 0$ and $\aul,\bul\in \B_m(\dul,\delta)$. We
have by part  $1$  of Lemma~\ref{TresLemas} that there exists
$n_0$ such that  for all $n\geq 0$
$$\B_{n+n_0}(\dul,\delta)\subset \B_n(\dul,\delta').$$
If $m\geq n_0$ then $\B_m(\dul,\delta) \subseteq
\B_{m-n_0}(\dul,\delta')$. Thus  we have
\begin{displaymath}
 |\phi(\aul)-\phi(\bul)|\leq w_r(\phi) r^{m-n_0}= \left(w_r(\phi)
r^{-n_0}\right)r^{m}.
\end{displaymath}
If  $1<m<n_0$, then  by  condition  $(\Path)$  we  have  that
Lemma~\ref{lem:path}
 is satisfied. So,  for all
$\aul,\bul\in \B_{m}(\dul,\delta)$,  there is a sequence
$\aul=\cul_0,\cdots,\cul_{\ell^m}=~\bul$  such that for all
$j=0,\cdots,\ell^m-1$ we have
$\varrho_m(\cul_j,\cul_{j+1})\leq\delta'$. Hence
$$|\phi(\aul)-\phi(\bul)|\leq\sum_{j=0}^{\ell^m-1}|\phi(\cul_j)-\phi(\cul_{j+1})|
\leq (\ell^m w_r(\phi))r^m.$$
\end{proof}
\subsection*{Topological Pressure}\label{PreTopological}
 From now on we assume that the metric $\varrho$  given in the introduction
satisfies $(\Conn)$ and that any  potential $\phi$ is weakly
H\"older continuous potential.

 Let  $\Lambda$  be a compact  subset of $X$
such that $\sigma(\Lambda)=\Lambda$, let  $\phi\in~\BV_r(X,\R)$
and $\aul\in X$. Set
$$ \displaystyle Z_n(\Lambda,\phi,\underline{a}):=\sum_{c^*\in \Z^n\\
 c^*\aul\in \Lambda}e^{S_n \phi (c^*\aul)}.$$
The standard  topological pressure of $\phi$  on $\Lambda$
  is given   by
\begin{equation}\label{pres}
P_{\Top}(\phi|_{\Lambda})= \lim_{n\rightarrow\infty}\frac{1}{n}
\log\left(Z_n(\Lambda,\phi,\underline{a})\right),
\end{equation}
see~\cite{PU10}).  Since $X$ is not compact,
 we define the
   \emph{ pressure} $P(\phi)$ of  a potential
   $\phi\in \BV_r(X,\R)$  on $X$
 with respect to $\sigma$ as follows
\begin{equation}\label{topol.press}
P(\phi):=\sup_{N\geq 1} P_{\Top}\left(\phi|_{\Sigma^+_N}\right),
\end{equation}
where for
 each $N\geq 1$
 $$P_{\Top}\left(\phi|_{\Sigma^+_N}\right):=\lim_{n\rightarrow\infty}
\frac{1}{n}\log\left(Z_n(\Sigma^+_N,\phi,\underline{a})\right).
$$ 
\begin{prop}
The following assertion holds
$$\sup\left\{P_{\Top}\left(\phi|_{\Lambda}\right):
 \Lambda\subseteq X  \textrm{ is compact, } \sigma(\Lambda)=\Lambda\right\}
=\sup_{N\geq 1 } P_{\Top}\left(\phi|_{\Sigma^+_N}\right).$$
\end{prop}
\begin{proof}
Let $\Lambda$ be a compact subset of $X$ such that
$\sigma(\Lambda)=\Lambda$. Then
 there exists $N\geq 1$ such that $\Lambda\subset\Sigma^+_N$ (see hypothesis of $X$ in the introduction).
 Hence for every
$\aul\in \Lambda$ we have
$$Z_n\left(\Lambda,\phi,\underline{a}\right)
\leq Z_n\left(\Sigma^+_N,\phi,\underline{a}\right).$$
  Therefore
$ P_{\Top}\left(\phi|_{\Lambda}\right)\leq\sup_{N\geq 1}
P_{\Top}\left(\phi|_{\Sigma^+_N}\right)$. Since it is satisfied
for every
 compact and  $\sigma$-invariant set $\Lambda$,  we have
$$P(\phi)=\sup \left\{P_{\Top}(\phi|_{\Lambda}):
 \Lambda\subseteq X  \textrm{ is  compact, } \sigma(\Lambda)=\Lambda\right\}.$$
\end{proof}
For the rest of this section  we assume that  each $\phi\in
\BV_r(X,\R)$ satisfies the following conditions
\begin{itemize}
\item[-]\emph{Summable} condition, that is
\begin{equation*}
\sup_{\aul\in
X}\left\{\sum_{\bul:\sigma(\bul)=\aul}\exp(\phi(\bul))\right\}
<\infty.
\end{equation*}
\item[-]
 \emph{Bounded on balls}. That is, if for all $R>0$ we
have
\begin{equation*}
\sup_{\cul\in B(\zul,R)}|\phi(\cul)|<\infty.
\end{equation*}
\end{itemize}
Thus, for every  $g\in \Cb_b(X,\R)$ we have that
 \emph{the transfer operator } $\mathscr{L}_{\phi}$  associated to the potential
 $\phi$ which is given by

 $$  \mathscr{L}_\phi(g)(\aul)=\sum_{\bul:\sigma(\bul)=\aul}e^{\phi(\bul)}g(\bul),
$$
 is well defined  on
 $\Cb_b(X,\mathbb{R})$,
  In
particular when $g$ is the function identically equal to $1$,
 we have
$$\mathscr{L}_\phi(\mathds{1})(\aul)=\sum_{\bul:\sigma(\bul)=\aul}e^{\phi(\bul)}.$$
Notice  that for every $n\geq 1$ and $\aul\in X $ the iterates of
 $\mathscr{L}_{\phi}$ are given by
\begin{displaymath}
\mathscr{L}^n_\phi(g)(\aul)=\sum_{c^*\in \Z^n}
e^{S_n \phi (c^*\aul)}g(c^*\aul).
\end{displaymath}
\begin{lem}\label{lem:ConsK}
Let  $\phi\in \BV_r(X,\mathbb{R})$. For every $n\geq 0$ and $m\geq
0$ there exists a constant  $K>0$  such that for all  $\dul\in X$
and $\aul,\bul\in\B_m(\dul,\delta)$ we have
\begin{displaymath}
\left|\mathscr{L}^n_{\phi}\mathds{1}(\bul)-\mathscr{L}^n_{\phi}\mathds{1}(\aul)\right|\leq
e^K\mathscr{L}^{n}_{\phi}\mathds{1}(\aul).
\end{displaymath}
\end{lem}
\begin{proof}
Let $n\geq 0$, $c^*\in \Z^n$ and let  $K=w_r(\phi_{n,c^*})$ be the
constant in  Lemma~\ref{lem:eq5}. Then for $\aul,\bul\in
\B_m(\dul,\delta)$ we have
  $$|e^{S_n \phi (c^*\bul)}-e^{S_n \phi (c^*\aul)}|
\leq e^K\left (e^{S_n \phi (c^*\bul)}\right).$$ Therefore,
 \begin{eqnarray*}
 |\mathscr{L}^n_{\phi}\mathds{1}(\bul)-\mathscr{L}^n_{\phi}\mathds{1}(\aul)|\leq\left|\sum_{c^*\in \Z^n}e^{S_n \phi (c^*\bul)}-
 \sum_{c^*\in \Z^n}e^{S_n \phi (c^*\aul)}\right|\\
\qquad\quad  \qquad\qquad\quad \leq\sum_{c^*\in
 \Z^n}\left|e^{S_n \phi (c^*\bul)}-e^{S_n \phi (c^*\aul)}\right|\\
\quad \quad\qquad\qquad\qquad  \leq e^K\sum_{c^*\in
 \Z^n}e^{S_n \phi (c^*\aul)}
 \leq e^K\mathscr{L}^n_{\phi}\mathds{1}(\aul).
 \end{eqnarray*}
 \end{proof}
\begin{lem}\label{eq:consK_}
For every potential $\phi\in \BV_r(X,\R)$  that is summable and
bounded on balls, and for every $R>0$ there exists $M_{\phi,R}>0$
such that
 for every $n\geq 0$ and  for all $\aul,\bul\in B(\underline{0},R)$ we have
 \begin{displaymath}
\mathscr{L}^n_{\phi}\mathds{1}(\aul)\leq
M_{\phi,R}\,\mathscr{L}^n_{\phi}\mathds{1}(\bul).
\end{displaymath}
\end{lem}
 \begin{proof}
Let $R>0$. From  $(\Conn)$ we have there exists $n_0=n_0(R)\geq 1$
such that for all $\aul,\bul\in B(\zul,R)$
 there exists $\cul\in \B(\aul,\delta_0)$ such
that $\sigma^{n_0}(\cul)=\bul$. Therefore, we have
\begin{eqnarray*}
\mathscr{L}^{n+n_0}_\phi\mathds{1}(\bul)=\sum_{\dul:\sigma^{n+n_0}(\dul)=\bul}e^{\phi_{n+n_0}(\dul)}=
\sum_{\dul:\sigma^{n+n_0}(\dul)=\bul}
e^{\phi_{n}(\dul)}.e^{\phi_{n_0}(\sigma^{n}\dul)}\\
\qquad\qquad \qquad\qquad\qquad\qquad
=\sum_{\eul:\sigma^{n_0}(\eul)=\bul}\sum_
{\dul:\sigma^{n}(\dul)=\eul}e^{S_n \phi (\dul)}.e^{\phi_{n_0}(\eul)}
\\\qquad\qquad\qquad \qquad \qquad\qquad\qquad \geq\sum_{\dul:\sigma^{n}(\dul)=\cul}e^{S_n \phi (\dul)}e^{\phi_{n_0}(\cul)}=e^{\phi_{n_0}(\cul)}\mathscr{L}^n_\phi\mathds{1}(\cul).
\end{eqnarray*}
 Observe that
\begin{eqnarray*}
\mathscr{L}^{n+n_0}_\phi\mathds{1}(\bul)=\sum_{\dul:\sigma^{n+n_0}(\dul)=\bul}e^{\phi_{n+n_0}(\dul)}=
\sum_{\dul:\sigma^{n+n_0}(\dul)=\bul}
e^{\phi_{n_0}(\dul)}.e^{\phi_{n}(\sigma^{n_0}\dul)}\\
\qquad\qquad\qquad\qquad\qquad=\sum_{\eul:\sigma^{n}(\eul)=\bul}\,
\, \sum_
{\dul:\sigma^{n_0}(\dul)=\eul}e^{\phi_{n_0}(\dul)}.e^{\phi_{n}(\sigma^{n_0}\dul)}
\\ \qquad\qquad \qquad\qquad \qquad\qquad \leq
\sum_{\eul:\sigma^{n}(\eul)=\bul}\left(e^{\phi_{n}(\eul)}\sum_{\dul:\sigma^{n_0}(\dul)=\eul}e^{\phi_{n_0}(\dul)}\right)
\\ \qquad \qquad\qquad\qquad\qquad\qquad\qquad\leq
\|\mathscr{L}_\phi\mathds{1}\|^{n_0}_\infty\mathscr{L}_\phi^n\mathds{1}(\bul).
\end{eqnarray*}
By Lemma~\ref{lem:ConsK} we have, for every $n\geq 0$,
$$\mathscr{L}_\phi^n\mathds{1}(\aul)\leq(1+e^K)
\mathscr{L}_\phi^n\mathds{1}(\cul).$$ Therefore,
\begin{multline*}
\mathscr{L}^{n}_\phi\mathds{1}(\aul)\leq
(1+e^K)e^{-\phi_{n_0}(\cul)}\mathscr{L}^{n+n_0}_\phi\mathds{1}(\bul)
\\\leq
(1+e^K)e^{-\phi_{n_0}(\cul)}\|\mathscr{L}_\phi\mathds{1}\|^{n_0}_\infty\mathscr{L}_\phi^n\mathds{1}(\bul)\\
\qquad\qquad\qquad\qquad\qquad\leq
(1+e^K)M_{\phi,R}\|\mathscr{L}_\phi\mathds{1}\|^{n_0}_\infty\mathscr{L}_\phi^n\mathds{1}(\bul),
\end{multline*}
where $M_{\phi,R}=\displaystyle\sup_{\cul\in B(\zul,R+\delta_0)}
e^{-\phi_{n_0}(\cul)}$, which is bounded by hypothesis.
 \end{proof}
\begin{cor} Let $\phi\in \BV_r(X,\R)$, then we have that $$\limsup_{n\rightarrow \infty}\frac{1}{n}\log
 \mathscr{L}_{\phi}^n\mathds{1}(\aul)$$ is independent of\, $\aul\in X$.
\end{cor}
\begin{proof}
Let $\aul,\bul\in X$, and set
$R=\max\{\varrho(\zul,\aul),\varrho(\zul,\bul)\}+1$. So, let
$M=M_{\phi,R}$  as in Lemma~\ref{eq:consK_}. It is such that for
all $n\geq 0$ we have
 \begin{displaymath}
 \mathscr{L}^n_{\phi}\mathds{1}(\aul)\leq
 M\mathscr{L}^n_{\phi}\mathds{1}(\bul).
 \end{displaymath}
 Hence
 $$\limsup_{n\rightarrow\infty}\frac{1}{n}\log  \mathscr{L}^n_{\phi}\mathds{1}(\aul)\leq
 \limsup_{n\rightarrow\infty}\frac{1}{n}\log \mathscr{L}^n_{\phi}\mathds{1}(\bul).$$
Thus the  corollary follows immediately.
 \end{proof}
We will show  that the sequence $\frac{1}{n} \log
\mathscr{L}^n_{\phi}\mathds{1}(\aul)$
 with $ \aul\in X $
  actually converges and its limit is precisely
$P(\phi).$
\subsection{The existence
of conformal measures} \label{Conf.Meas.}
 Throughout this section we consider the following notation. For any
$a^*=a_0 \cdots a_{n-1}\in \Z^n$, the cylinder $[a^*]$ is defined
by
$$[a^*]:=\{\bul\in X: b_i=a_i, 0\leq i\leq n-1\}.$$

For topologically mixing Markov shift spaces  with infinitely  many symbols, the  the existence of
conformal measures  is
 not very 
immediate, see for instance~\cite{MU01}. Due our space $X$ is a bit more complicated due the lack of compactness and since fall out of the framework developed by \cite{MU01}, we have  to impose on $X$ further assumptions to   state  the existence of
conformal measures. 
\begin{itemize}
\item[-]Assume that the transformation
 $\sigma:X\to
X$ can be continuously extended to
$\overline{\sigma}:\overline{X}^\varrho\to
\overline{X}^{\varrho}$, where $\X^{\varrho}$ is the completion of
$X$ with respect to the metric $\varrho$. \item[-]
 For all $R>0$, we denote by $\overline{B}(\underline{0},R)$ the
closed ball in $\X^{\varrho}$ centered in $\underline{0}$ \, and
radius $R$. Namely
$$\overline{B}(\underline{0},R):=\{\bul\in \overline{X}^\varrho:
\varrho(\underline{0},\bul)\leq R\},$$
 and assume that this ball $\overline{B}(\underline{0},R)$ is compact in $\X^{\varrho}$
\item[-]
 We also  assume  that
for every $R>0$ and for every $\bul\in \X^{\varrho}\sms X$ there
exists $N\geq 1$ such that for $n\geq N$, we have
$\overline{\sigma}^n(\bul)\in \X^{\varrho}\sms
\overline{B}(\underline{0},R)$. If it does not lead to
misunderstanding we will frequently  denote $\overline{\sigma}$
simply by $\sigma$. \item[-] Given $R>0$ and $k$ be an integer
number. Let
 \begin{quote}
$[k,R]:=[k]\cap \overline{X}^\varrho\backslash
\overline{B}(\underline{0},R)=\{\aul\in [k]\cap
\overline{X}^\varrho:\varrho(\underline{0},\aul)>~R\}$.
\end{quote}
 And given a potential $\phi\in
\BV_r(X,\mathbb{R})$, we assume that $\phi$ satisfies
\begin{equation*}
\sum\limits_{k\in \Z}e^{\sup \phi |_{[k ]}}<\infty.
\end{equation*}
\begin{equation*}
\sum\limits_{k\in\Z}e^{\sup\phi |_{[k,R]}}\rightarrow 0,\textrm{
when } \,R\rightarrow \infty.
\end{equation*}
\item[-] Finally we assume that $\phi$ has a continuous extension
on $\X^{\varrho}$.
\end{itemize}

\medskip

Let us recall what  a conformal measure  means. Consider a
measurable endomorphism $T:Y\rightarrow Y$ on a measurable space
$(Y,\mathscr{F})$ and a measurable non-negative function $g$ on
$Y.$ A measure $m$ on $(Y,\mathscr{F})$ is called $g$-conformal
for $T$ on $Y$ if for all  measurable set $A$ which $T(A)$ is
measurable and $T|_{A}$ is invertible we have
\begin{equation}\label{eq:Conformal}
m(T(A))=\int_A g dm.
\end{equation}
Observe that (\ref{eq:Conformal}) implies that $m\circ T$ is
absolutely continuous with respect to $m$ on the $\sigma$-algebra
$\mathscr{F}\cap A$, for every set $A\in\mathscr{F}$ such that
$T:A\rightarrow~T(A)$ is a measurable isomorphism. In this case,
the equality in~(\ref{eq:Conformal})is equivalent to the fact that
the
 corresponding
Radon-Nikodym derivative $\frac{d (m\circ T)|_A}{dm|_A}$ is equal
 to $g|_A$.

\medskip

 Let $C(X,\R)^*$ be the dual space of $C(X,\R)$, $\phi\in
 \BV_r(X,\R)$ and
 $\mathscr{L}^*_\phi:C(X,\R)^*\rightarrow C(X,\R)^*$ the dual
 Ruelle transfer operator defined by
 $$
 \mathscr{L}^*_\phi(\nu)(g)= \nu(\mathscr{L}_\phi(g)).
 $$
 For every $N\geq 1 $ we have that  $\Sigma_N^+$ is a compact  metric
 space in $X$. So the
\emph{Schauder-~Tychonoff }
 fixed point theorem
 guarantees that exists a
fixed point $\nu_{_N}$ (supported  on  $\Sigma^+_N$) of the map
$$\nu\rightarrow(\mathscr{L}_{\phi}^*(\nu)(\uno))^{-1}
\mathscr{L}^*_\phi(\nu).$$ That is
$$(\mathscr{L}_{\phi}^*\nu_{_N}(\uno))^{-1}
\mathscr{L}^*_\phi\nu_{_N}=\nu_{_N}.$$ Following the classical
thermodynamic formalism for compact spaces, It is known that for
each $N$  the fixed point $\nu_{_N}$ is a
$e^{P_{\Top}(\phi|_{\Sigma^+_N})-\phi|_{\Sigma^+_N}}$-conformal
measure for $\sigma|_{\Sigma^+_N}$, see~
\cite{PU10}. It means  that  for each  $N$  we have
 \begin{equation}
\nu_{_N}(\sigma([k,R]_N))=e^{P_{\Top}(\phi|_{\Sigma^+_N})}
\int_{[k,R]_N}e^{-\phi|_{_{\Sigma^+_N}}}d\nu_{_N}(\aul),
\end{equation}
where  $[k,R]_N= \Sigma^+_N\cap [k,R] $ such that
$\sigma|_{_{[k,R]_N}}$ is invertible
\begin{rem}\label{rem:eingevalue}
Since for every  $N\geq 1$ we have
  $ P_{N+1}(\phi)\geq
  P_{N}(\phi)$,
and  $  \lim_{N\rightarrow \infty}
P_{\Top}(\phi|_{\Sigma^+_N})=P(\phi) $, we have that
  $\lim_{N\rightarrow\infty}e^{P_{\Top}(\phi|_{\Sigma^+_N})}$ exists and is equal to
  $e^{P(\phi)}$. We will denote    $e^{P(\phi)}$ by $\lambda.$
\end{rem}
 Since our goal in this section is to construct conformal measures  without
 the condition of compactness,
 we  first show that the family $\{\nu_{_N}\}_{N\in \N}$  is tight  and then to
apply Prokhorov's  Theorem. We recall  that \emph{tightness} means
that
\begin{quote}
for all $\varepsilon>0$ there exists a compact set $K\subset
\X^{\varrho}$ such
  that for all $N$ we have $\nu_{_N}(\X^{\varrho}\backslash K)<\varepsilon$.
\end{quote}
\begin{prop}
  The sequence of measures $\{\nu_{_N}\}_{N\in \N}$ is \emph{tight}.
\end{prop}
\begin{proof}
Since  for each $N$ the measure $\nu_{_N}$ is
$e^{P_{\Top}(\phi|_{\Sigma^+_N})-\phi|_{\Sigma^+_N}}$-conformal,
we have
  \begin{eqnarray*}
    \nu_{_N}(\sigma([k,R]))=e^{P_{\Top}(\phi|_{\Sigma^+_N})}\int_{[k,R]}e^{-\phi}d\nu_{_N}(\aul)\\
\qquad\qquad\qquad\qquad
    \geq
    e^{P_{\Top}(\phi|_{\Sigma^+_N})}\nu_{_N}([k,R])e^{-\sup\phi|_{[k,R]}}.
  \end{eqnarray*}
 Hence
$$\nu_{_N}([k,R])\leq e^{-P_{\Top}(\phi|_{\Sigma^+_N})}\nu_{_N}(\sigma([k,R]))e^{\sup\phi|_{[k,R]}}
\leq e^{-P_{\Top}(\phi|_{\Sigma^+_N})}e^{\sup\phi|_{[k,R]}}.$$
  Furthermore  observe that
  \begin{multline}\label{eq:100}
\nu_{_N}(X\sms\overline{B}(0,R))=\nu_{_N}\left(\bigcup_{k\in
\Z}[k,R]\right)\\ \leq \sum_{k\in \Z}\nu_{_N}([k,R])\leq
    e^{-P_{\Top}(\phi|_{\Sigma^+_N})}\sum_{k\in \Z}\
    e^{\sup\phi|_{[k,R]}}
  \end{multline}
 From (\ref{eq:Cond.Potential2}) we have that  the last term tends
  to  zero when $R$ tents to infinity.
  Therefore from  the compactness  of  the  closed ball
  $ \overline{B}(0,R)$
 in $\X^{\varrho}$ the
  tightness of $\{\nu_{_N}\}_{N\geq 1}$ is proved.
\end{proof}
The Prokhorov's theorem 
allows us relate the
tightness of measures to weak convergence in the space of
probability measures.
\begin{teo}[Prokhorov] If $P$ is a Polish space, that means  a
complete metrizable and separable space, and  $\mathcal{M}(P)$ is
the space of all  Borel
 probability measures in $P,$ then every
\emph{tight} family measures from $\mathcal{M}(P)$ is a
pre-compact subset of $\mathcal{M}(P)$.
\end{teo}
 Therefore, for the  sequence of measures $\{\nu_{_N}\}_{N\geq 1}$
 we have,
there exists a subsequence $ \{\nu_{_{N_i}}\}_{i\geq 1}$  which
converges in  the weak topology to some probability measure $\nu$.
It follows  in particular that  for every  Borel set $A$ such that
$\nu(\partial{A})=0$ and for every   bounded continuous
 function $g$ with bounded support  we have
\begin{equation}
\lim_{i\rightarrow\infty}\int_A g d\nu_{_{N_i}}=\int_A g d\nu.
\end{equation}
Let
\begin{displaymath}
  X_{\rad}:=\{\aul\in \overline{X}^\varrho: \omega_\sigma(\aul)\cap X\neq \emptyset\}.
\end{displaymath}
 For every $N\geq 1$, let  $\overline{B}(0,N)$  be the closed ball  in $\X^{\varrho}$
and for  $k$ an integer number  let $\overline{[k]}$ be  the
closure of  $[k].$ Thus we consider  a sequence of $A_N$, where
$$A_N:=\overline{B}(0,N)\cap \bigcup_{k=-N}^{N}(\overline{[k]}).$$
For every $N\geq 1$ we have $A_N\subset A_{N+1}$ and
$\bigcup_{N=1}^{\infty} A_N= \X^{\varrho}.$
\begin{teo}\label{conformalmeasure}
The measure $\nu$ is $ e^{P(\phi)-\phi}$-conformal and
$\nu(X_{\rad})=1$.
\end{teo}
\begin{proof}
Note that   for each $N\geq 1$  the measure  $\nu_{_N}$ is
$e^{P_{\Top}(\phi|_{\Sigma^+_N})-\phi|_{\Sigma^+_N}}$-conformal
for $\sigma|_{\Sigma^+_N}$ but not for $\sigma$ defined on
$\X^{\varrho}$. However,  if  $N$  large  enough
 then for every  $A\subset A_N$
 such that $\sigma|_{A}$ is one-to-one, we have
  \begin{displaymath}
    \nu_{_N}(\sigma(A))=e^{P_{\Top}(\phi|_{\Sigma^+_N})}\int_A e^{-\phi}
    d\nu_{_N}.
  \end{displaymath}
To verify  this, first  we prove that
\begin{equation}\label{Ec:sigma}
\sigma(A)\cap\Sigma^+_N=\sigma(A\cap\Sigma^+_N).
\end{equation}
Observe  that $\sigma(A\cap\Sigma^+_N)\subseteq
\sigma(A)\cap\sigma(\Sigma^+_N) \subseteq\sigma(A)\cap\Sigma^+_N.$
To get the contrary   inclusion, let
 $\aul\in\sigma(A)\cap\Sigma^+_N, $ then
there exists $\bul \in A$ such that $\aul=\sigma(\bul)$. We prove
that  $\bul\in \Sigma^+_N$.
 Let  $c \in \Z$  a word
such that  $c \aul=\bul$.  Then
 since $\bul\in A$ and $A\subset A_N$  we have that  there exists $k\in \{-N,\cdots, N\}$
 such that $c \aul\in \overline{[k]}$,
 then
 we have
$$ \B_1(c\aul,\delta)\cap[k]\neq \emptyset.$$
Thus  by part $1$ of Lemma~\ref{TresLemas} we have that
$\B_1(c\aul,\delta)\subseteq c\B_0(\aul,\delta)$ and   hence
$$c\B_0(\aul,\delta)\cap[k]\neq \emptyset.$$
Therefore $c\in \{-N,\cdots, N\}$ and  so
 we obtain $\bul\in \Sigma^+_N.$  Using (\ref{Ec:sigma})  we can write
\begin{eqnarray*}
\nu_{_N}(\sigma(A))=
\nu_{_N}(\sigma(A)\cap\Sigma^+_N)\\
\quad\quad\quad\quad\quad= \nu_{_N}(\sigma(A\cap\Sigma^+_N))
\\
\quad\quad\quad\quad\quad=
e^{P_{\Top}(\phi|_{\Sigma^+_N})}\int_{A\cap\Sigma^+_N} e^{-\phi}
    d\nu_{_N}\\
\quad\quad\quad\quad\quad =\nu_{_N}(\sigma(A))=
e^{P_{\Top}(\phi|_{\Sigma^+_N})}\int_A e^{-\phi}
    d\nu_{_N}
 \end{eqnarray*}
 And since the sequence  $\{\nu_{_{N_i}}\}$ converges weakly to $\nu$ we
  have that  for every
 Borel set $A$ such that  $\nu(\partial A)=0$ it satisfies
$\nu_{_{N_i}}(A)\rightarrow \nu(A).$ In particular
  this holds for every bounded Borel A such that $\nu(\partial A)=0$
  and $\nu(\partial \sigma(A))=0.$
  Then   \begin{displaymath}
  \nu(\sigma(A))=\lim_{i\rightarrow\infty}\nu_{_{N_i}}(\sigma(A))=
  \lim_{i\rightarrow\infty}e^{P_{\Top}(\phi|_{\Sigma^+_N})}\int_A e^{-\phi}
  d\nu_{_{N_i}}=\lambda\displaystyle\int_A e^{-\phi}d\nu.
  \end{displaymath}
 Take an arbitrary bounded Borel set $A$ such that $\sigma|_A$ is
 injective.
   We have   that $\Sing(\sigma:\X^{\varrho}\rightarrow\X^{\varrho})=\emptyset$ and $\nu(\X^{\varrho})=1$. Then   we obtain that  $\nu$ is a
   conformal measure.

 To prove that  the measure $\nu$ is supported in $X_{\rad}, $ note that by the inequality
 (\ref{eq:Cond.Potential2}), there exists $R>0$ such that
   \begin{equation}\label{eq:110}
     \lambda^{-1}\sum_{k\in \Z}e^{\phi|_{[k,R]}}<1/2.
   \end{equation}
   Let
   \begin{displaymath}
     \X^{\varrho}(R,n):=\{\aul\in \X^{\varrho}:
     (\overline{\sigma})^k(\aul)\in \X^{\varrho}\sms B(\underline{0},R)\textrm{ for }
     k=0,\ldots,n-1\}.
   \end{displaymath}
   Then
   \begin{eqnarray*}
     \nu(\X^{\varrho} (R,n))\geq \nu\big(\overline{\sigma}(\X^{\varrho}(R,n+1))\cap [k])\big)\\
     =\lambda\int_{\X^{\varrho} (R,n+1))\cap[k]}e^{-\phi}d\nu\geq
     \lambda \nu(\X^{\varrho} (R,n+1))\cap [k])e^{-\sup\phi|_{\X^{\varrho} (R,n+1))\cap [k]}}.
   \end{eqnarray*}
   Hence using (\ref{eq:110}), we get that
   \begin{displaymath}
     \nu(\X^{\varrho} (R,n))\leq \lambda^{-1}\sum_{k\in \Z}
    e^{\sup\phi|_{[k,R]}}\nu(\X^{\varrho} (R,n-1)),
   \end{displaymath}
   and then
   \begin{displaymath}
     \nu(\X^{\varrho}(R,n))\leq (1/2)^n.
   \end{displaymath}
   Thus $\nu(X_{\rad})=1$.
\end{proof}
\begin{cor}
  There exists a measure $\nu$ which is $e^{P(\phi)-\phi}$-conformal for
  $\sigma:\overline{X}^\varrho\rightarrow \overline{X}^\varrho$ and $\nu(X_{\rad})=1$.
\end{cor}
For the rest of this section  we  assume that $\phi\in
B_{r}(X,\mathbb{R})$ is \emph{summable} and \emph{bounded on
balls} see~(\ref{sumable}) and (\ref{balls}).

We say that  $\phi$ is \emph{rapidly decreasing} if it satisfies.
 \begin{equation*}
\lim_{R\rightarrow\infty}\sup_{\aul\in
\overline{X}^\varrho\backslash
B(\zul,R)}\big(\mathscr{L}_\phi\mathds{1}(\aul)\big)=0.
\end{equation*}
We consider the normalized transfer operator
$\widehat{\mathscr{L}}_\phi= e^{-P(\phi)}\mathscr{L}_\phi.$
\begin{prop}\label{Prop:pressure}
Let $\phi$ be a rapidly decreasing  potential. The limit
$\lim_{n\rightarrow\infty}\frac{1}{n}\log
\mathscr{L}^{n}_\phi\mathds{1}(\aul)$, with  $\aul\in
\overline{X}^\varrho$ exists, and
$$P(\phi)=\lim_{n\rightarrow\infty}\frac{1}{n}\log
\mathscr{L}^{n}_\phi\mathds{1}(\aul).$$
\end{prop}
\begin{proof}
To prove this proposition is enough  to  prove that there exists
$L>0$ and, for all $R>0$ there exists $\ell_R>0$
 such that for all $n\geq 1$ and all $\aul\in B(0,R)$, we have
\begin{equation}\label{inequality}
 \ell_R\leq\widehat{\mathscr{L}}^{n}_\phi\mathds{1}({\aul})\leq
L.
\end{equation}

 First, we prove the right hand side inequality.  Since
$\phi$ is rapidly decreasing  
 we have  that  there exists  sufficiently large $R_0>0$  that
\begin{equation}\label{eq:lessthanone}
\sup\left\{\widehat{\mathscr{L}}_\phi\mathds{1}({\aul}):
\aul\in\overline{X}^\varrho\backslash B(\zul,R_0) \right\}\leq 1.
\end{equation}
We will show by induction that for every $n\geq 0, $
\begin{equation}\label{eq:lessthanL}
\|\widehat{\mathscr{L}}^{n}_\phi\mathds{1}\|_\infty\leq
\frac{M_{\phi,R_0}}{\nu(B(\zul,R_0))}:=L,
 \end{equation}
where $M_{\phi,R_0}$ is the constant coming from
Lemma~\ref{eq:consK_} with $\mathscr{L}_\phi$ replaced by the
operator $\widehat{\mathscr{L}}_\phi$.

  To see this observe that  for all $n\geq 0$  it follows
\begin{equation}
\|\widehat{\mathscr{L}}^n_\phi\mathds{1}\|_{\infty}\leq
e^{-nP(\phi)}\|\mathscr{L}_\phi\mathds{1}\|^n_{\infty}.
\end{equation}
Thus for $n=0$ is clear, since
$$\|\widehat{\mathscr{L}}^0_\phi\mathds{1}\|_{\infty}\leq
1\leq M_{\phi,R_0}\leq\frac{M_{\phi,R_0}}{\nu(B(0,R_0))}.
$$
Suppose that
$$\|\widehat{\mathscr{L}}^n_\phi\mathds{1}\|_{\infty}\leq\frac{M_{\phi,R_0}}{\nu(B(0,R_0))}.
$$
 We will now  prove this inequality  for $n+1$.
Using (\ref{decreasing}) and the fact that  $\phi$ is summable
 we have,
 there exists $\bul\in \overline{X}^\varrho$ such that
\begin{displaymath}
\|\widehat{\mathscr{L}}^{n+1}_\phi\mathds{1}\|_{\infty}=
\widehat{\mathscr{L}}^{n+1}_\phi\mathds{1}(\bul).
\end{displaymath}
If  $\bul\in \overline{X}^\varrho\backslash B(0,R_0)$,  then by
(\ref{eq:lessthanone}) and (\ref{eq:lessthanL})
\begin{eqnarray*}
\|\widehat{\mathscr{L}}^{n+1}_\phi\mathds{1}\|_\infty=
\widehat{\mathscr{L}}^{n+1}_\phi\mathds{1}(\bul)=
\widehat{\mathscr{L}}_\phi(\widehat{\mathscr{L}}^{n}_\phi\mathds{1})(\bul)\\
\qquad\qquad\qquad\leq
e^{-P(\phi)}\sum_{\cul\in\sigma^{-1}(\bul)}e^{\phi(\cul)}\widehat{\mathscr{L}}^{n}_\phi\mathds{1}(\cul)
\\
\qquad\qquad\qquad\qquad\leq
\|\widehat{\mathscr{L}}^{n}_\phi\mathds{1}\|_\infty\widehat{\mathscr{L}}_\phi
\mathds{1}(\bul)
\\
\qquad\qquad\qquad\qquad\qquad\leq
 \frac{M_{\phi,R_0}}{\nu(B(0,R_0))}.\\
\end{eqnarray*}
Otherwise, if $\bul\in B(0,R_0)$,  then by
 Lemma~\ref{eq:consK_} and the fact that  for all $n\geq 0$,
$\widehat{\mathscr{L}}^{n}_{\phi}\mathds{1}\nu=\nu$
 we follow
\begin{eqnarray*}
1= \int d\nu = \int
\widehat{\mathscr{L}}^{n+1}_\phi\mathds{1}d\nu\\
\qquad\qquad\qquad\geq
\int_{B(\zul,R_0)}\widehat{\mathscr{L}}^{n+1}_\phi\mathds{1}d\nu
\geq
M^{-1}_{\phi,R_0}\nu(B(\zul,R_0))\|\widehat{\mathscr{L}}^{n+1}_\phi\mathds{1}\|
_{\infty}.
\end{eqnarray*}
 Thus the  right hand inequality is proved.

Now we  prove  the other inequality in~(\ref{inequality}). Let
$R_1>R_0$ be  such that
 $\nu(\overline{X}^\varrho\backslash B(\zul,R_1))\leq 1/4L$.
Let $R>R_1$ then we have
\begin{eqnarray*}
1=\int\widehat{\mathscr{L}}^n_\phi\mathds{1} d\nu \leq
\int_{B(\zul,R)}\widehat{\mathscr{L}}^{n}_\phi\mathds{1}
d\nu+\int_{\overline{X}^\varrho\backslash
 B(\zul,R)}\widehat{\mathscr{L}}^{n}_\phi\mathds{1} d\nu
\\
\qquad\qquad\qquad\qquad\leq
\int_{B(\zul,R)}\widehat{\mathscr{L}}^{n}_\phi\mathds{1} d\nu+ L
\,\nu(\overline{X}^\varrho\backslash  B(\zul,R))
\\
\qquad\qquad\qquad\qquad\qquad\leq\int_{B(\zul,R)}\widehat{\mathscr{L}}^{n}_\phi\mathds{1}
d\nu+1/4.
\end{eqnarray*}

Hence for any $n\geq 0$ there exists $\aul_n\in B(\zul,R_0)$
 such that $\widehat{\mathscr{L}}^{n}_{\phi}\mathds{1}(\aul_n)
\geq 3/4$. If $\cul\in B(\zul,R)$ is any other point, then by the
 Lemma~\ref{eq:consK_} we have
 $$3/4\leq\mathscr{L}^{n}_\phi\mathds{1}(\aul_n)\leq
M_{\phi,R}\mathscr{L}^{n}_\phi\mathds{1}(\cul).$$
 Thus, for each $R\geq R_1 $ the inequality holds with
$l_R=3M^{-1}_{\phi,R}/4$,  if $R<R_1$ then  we just put
$l_R:=l_{R_1}$.
 Therefore $\lim_{n\rightarrow\infty}\frac{1}{n}\log
\mathscr{L}^n_{\phi}\mathds{1}(\aul)$ exists and it is equal to
$P(\phi)$.
\end{proof}
\subsection{The existence of an invariant probability Gibbs measure}
\label{Gibbs Measures}
In this section we prove for weakly H\"older potentials
 the existence  of an invariant measure $\mu$
which is absolutely continuous with respect to the conformal
measure $\nu$ (given in Section~\ref{Conf.Meas.}) and we also
prove that $\mu$ is a Gibbs measure. We give the definition of
Gibbs measures of weakly H\"older potentials, with a definition
appropriately adapted
 for the transcendental functions.   A probability  measure $\eta$ on $X$  and $\mathscr{F}$ the
Borel $\sigma$-algebra of
  sets   is called   a \emph{Gibbs measure} for  a  weakly H\"older potential  $\phi
  $  if  there exist $\Pe\in \mathbb{R}$,  $C\geq 1$ such that  for all
 $\aul \in X$, there exists  $M=M(\aul)$ such that  for
all $n\geq 1$ and $c^*\in \Z^n$ we have
\begin{equation}\label{Gibbs}
C^{-1}M(\aul)\leq \frac{\eta(c^*\B_0(\aul,\delta
))}{\exp(S_n \phi (c^*\aul)- n\Pe)}\leq C.
\end{equation}
If additionally  $\nu$ is $\sigma$-invariant,
 we call $\nu$  \emph{ an invariant Gibbs  measure}.
\begin{lem}\label{positive} Let  $\nu$ be the
 conformal  measure
 from Theorem~\ref{conformalmeasure}. Then for every
 $\aul \in X$ and for every  $r>0$,
 we have $\nu(B(\aul,r))>0.$
\end{lem}
\begin{proof}
Let $\aul\in X$ and $r>0$. Since
 $X=\bigcup_{R\geq 0}B(\zul,R)$
 we have  there exists $R>0$ such that
$\nu(B(0,R))>0.$ By Lemma~\ref{aul_r}  there exists $n_0>0 $ such
that
$$B(\sigma^{n_0}\aul,\delta_0)\subseteq
\sigma^{n_0}(B(\aul,r)).$$ Take $R_1>0$ such that  $R_1\geq R$ and
$\sigma^{n_0}\aul\in B(\zul, R_1)$. By $(\Conn)$ property
 we have  there exists
 $n_1$ such that
$$ B(\zul,R_1)\subseteq \sigma^{n_1}(B(\sigma^{n_0}\aul,\delta_0))
\subseteq
 \sigma^{n_1+n_0}(B(\aul,r)).$$
 Thus  for $n=n_0+n_1$  we have
$B(\zul,R)\subseteq \sigma^{n}(B(\aul,r)).$

Since $B(\aul,r)= \bigcup_{c^*\in \Z^n}B(\aul,r)\cap [c^*]$. Then
\begin{displaymath}
\qquad\qquad\qquad \sigma^{n}(\bigcup_{c^*\in \Z^n} B(\aul,r)\cap
[c^*])\supset
 B(0,R).
\end{displaymath}
  Then there exists $c^*\in \Z^n$ such that
 $$\nu(\sigma^{n}(B(\aul,r)\cap[c^*]))>0.$$
 Therefore, since $\sigma^{n}|_{B(\aul,r)\cap[c^*]}$
 is one-to-one, we have
\begin{eqnarray*}
0<\nu(\sigma^{n}(B(\aul,r)\cap [c^*]))=
  \int_{B(\aul,r)\cap[c^*]}
  e^{n P(\phi)-S_n \phi }d\nu\\
\qquad\qquad\qquad\qquad\qquad\qquad \leq e^{n P(\phi)}
e^{-\inf{\phi_{n}|_ {B(\aul,r)\cap[c^*]}}}
\nu(B(\aul,r)\cap[c^*])\\
\qquad\qquad\qquad\qquad\qquad\qquad\qquad \leq e^{n
P(\phi)}e^{-S_n \phi (\aul)}C\nu (B(\aul,r)),
\end{eqnarray*}
 where $C$ is the positive constant $\exp\left(\frac{w_r(\phi)}
{1-r}\right)$ from Lemma~\ref{lem:varrho}. Thus
$$ \nu(B(\aul,r))\geq \nu(\sigma^{n}B(\aul,r)\cap [c^*])
e^{-n P(\phi)} e^{S_n \phi (\aul)}C>0.
$$
\end{proof}
\begin{prop}
Let $\phi\in \BV_r(X,\R)$ satisfying the  conditions $(\ie)$
$(\iie)$, $(\iiie),$ $(\ive)$ and $(\ve)$. Then
  there exists an invariant measure $\mu$
 that is
absolutely continuous with respect to
  $\nu$ and it is a Gibbs measure  for $\phi.$

\end{prop}
\begin{proof}
    Let $A\subset B(\aul,\delta)\subset B(\zul,R)$ be a Borel set.
    By the conformality of the measure $\nu$, we have that
    \begin{equation}
        \label{Bar:1}\nu(A)=
    \int_{c^*A} e^{nP(\phi)}e^{-S_n \phi }d\nu
    \end{equation}
    and by Lemma \ref{eq:consK_} and again the conformality, we also have
    \begin{equation}
        \label{Bar:2}
        \sum_{c^*\in \Z^n}e^{-nP(\phi)} e^{S_n \phi (c^*\aul)}\geq
        M_{\phi,R}^{-1}e^{-nP(\phi)} \int_{B(\zul,R)}\mathscr{L}^n\mathds{1}
        d\nu=M_{\phi,R}^{-1}\nu(B(\zul,R)).
    \end{equation}
    Then, by (\ref{phi_nc}) and (\ref{Bar:1}),
    \begin{equation}
        \label{Bar:3}
        \nu(A)
        \geq e^{nP(\phi)}
        e^{-\sup(S_n \phi |_{c^*\B_0(\aul,\delta))}}
        \nu(c^*A)
        \geq e^{nP(\phi)} e^{-S_n \phi (c^*\aul)}
        C^{-1}\nu(c^*A)
    \end{equation}
    and
    \begin{equation}
        \label{Bar:4}
        \nu(A)
        \leq e^{nP(\phi)}
        e^{-\inf(S_n \phi |_{c^*\B_0(\aul,\delta)})}
        \nu(c^*A)
        \leq e^{nP(\phi)} e^{-S_n \phi (c^*\aul)}
        C\nu(c^*A).
    \end{equation}
    In particular
    \begin{displaymath}
        e^{-nP(\phi)} e^{S_n \phi (c^*\aul)}
        \leq \frac{C\nu(c^*B(\aul,\delta))}{\nu(B(\aul,\delta))}.
    \end{displaymath}
    Then by (\ref{Bar:3}), we have
    \begin{multline*}
        \nu(\sigma^{-n}(A))=\sum_{c^*\in \Z^n}\nu(c^*A)\leq \nu(A)
        C\sum_{c^*\in \Z^n}e^{-nP(\phi)} e^{S_n \phi (c^*\aul)}\\
        \leq \nu(A)C^2/\nu(B(\aul,\delta)),
     \end{multline*}
     and by (\ref{Bar:4}) and (\ref{Bar:2})
    \begin{multline*}
        \nu(\sigma^{-n}(A))=\sum_{c^*\in \Z^n}\nu(c^*A)\geq \nu(A)
        C^{-1}\sum_{c^*\in \Z^n}e^{-nP(\phi)} e^{S_n \phi (c^*\aul)}
        \\\geq \nu(A)(CM_{\phi,R})^{-1}\nu(B(\zul,R)).
     \end{multline*}
     Therefore,
     \begin{equation}
        \label{Bar:5}
        \nu(A)(CM_{\phi,R})^{-1}\nu(B(\zul,R))\leq \nu(\sigma^{-n}(A))
        \leq \nu(A)C^2/\nu(B(\aul,\delta))
     \end{equation}

    In the Banach space  of all bounded sequences of real numbers
    we consider  a continuous  linear functional $L$ called a
    \emph{Banach limit} $L~:~\ell^{\infty}\rightarrow~\mathbb{R}$
    (see~\cite{BK85}) such that
    \begin{enumerate}
    \item[-]$\li(\{\aul_{n}\}_{n\geq 1})=\li(\{\aul_{n+1}\}_{n\geq
        1})$.
    \item[-]$\liminf_{n\rightarrow\infty} a_n \leq \li(\{a_n\})\leq \limsup_{n\rightarrow \infty} a_n$
    \item[-]If $\{a_n\}$ converges then $\lim_{n\rightarrow \infty}
        a_n=\li(\{a_n\}).$
    \end{enumerate}
    Then let $\nu$ be a measure defined by formula
    $$\mu(A)= \li(\{\nu\circ\sigma^{-n}(A)\}_{n=0}^{\infty}).$$
    Since $\li(\{\aul_{n}\}_{n\geq 1})=\li(\{\aul_{n+1}\}_{n\geq 1})$,
    the measure is invariant. Moreover, since $\liminf_{n\rightarrow\infty} a_n \leq \li(\{a_n\})\leq \limsup_{n\rightarrow\infty} a_n$, by (\ref{Bar:5})
    we get that for any Borel set $A\subset B(\aul,\delta)\subset
    B(\zul,R)$, we have
    \begin{equation}\label{bar11}
\nu(A)(CM_{\phi,R})^{-1}\nu(B(\zul,R))\leq \mu(A)
        \leq \nu(A)C^2/\nu(B(\aul,\delta))
        \end{equation}

    Hence the measure $\mu$ is equivalent to the measure $\nu$.
    \end{proof}
%

\begin{prop}
The measures
 $\nu$ and $\mu $ are
Gibbs measure.
\end{prop}
\begin{proof}
  From (\ref{Bar:3}) and (\ref{Bar:4}) we  have
  \begin{equation}\label{bar12}
    C^{-1}\nu(B(\aul,\delta))\leq
\frac{\nu(c^*\B(\aul,\delta))}{\exp(S_n \phi (c^*\aul)-n
P(\phi))}\leq C.
  \end{equation}
  Using
%
Lemma~\ref{positive} follows that   for every $\aul\in X$,
$\nu(B(\aul,\delta))>0$. Therefore  we get that $\nu$ is  a Gibbs
measure with $P=P(\phi)$, $C=\exp(\frac{w_r(\phi)}{1-r})$ and
$M(a)= \nu(B(\aul,\delta_0)).$

 We also have that from (\ref{bar11}) and (\ref{bar12})
$$
C^{-2} M_{\phi,R}^{-1} M(a)^2\leq
\frac{\mu(c^*\B(\aul,\delta))}{\exp(S_n \phi (c^*\aul)-n P(\phi))}
\leq C^3.
$$
Therefore $\mu$ is a Gibbs measure.
\section*{Acknowledgements}
The author would like to thank  the valuable suggestions and comments of Lasse Rempe and Gerardo Honorato. This was partially supported by 
   Fondecyt 1130341. 
\end{proof}%


\begin{thebibliography}{10}


\bibitem[Ba08]{Ba08}
K.~Bara{\'n}ski, \emph{Hausdorff dimension of hairs and ends for entire maps of finite order}, Math. Proc. Cambridge Philos. Soc.. \textbf{145} (2008), no.~3, 719--737.

\bibitem[Ba07]{Ba07}
K.~Bara{\'n}ski, \emph{Trees and hairs for some hyperbolic entire maps of finite order}, Math. Z. \textbf{257} (2007), no.~1, 33--59.


\bibitem[BJR12]{BJR12}
K. Barański and X. Jarque and L. Rempe,
\emph{Brushing the hairs of transcendental entire functions}, Topology and its Applications, (2012), no.~8,  159, 2102--2114.



\bibitem[RRRS11]{RRRS11}
G. Rottenfusser,  J. Ruckert,  L. Rempe and D. Schleicher,
\emph{Dynamic rays of bounded-type entire functions},
Annals of Mathematics,  (2011 )Pages 77-125, 173.













  
	






  
 
 \bibitem[BK07]{BK07}
K.~Bara{\'n}ski and B.~Karpi{\'n}ska, \emph{Coding trees and boundaries of
  attracting basins for some entire maps}, Nonlinearity \textbf{20} (2007),
  no.~2, 391--415.

\bibitem[BK85]{BK85}
A.~Brunel and U.~Krengel, \emph{Ergodic theorems}, Berlin; New York: W. de
  Gruyter, 1985.





\bibitem[CS07]{CS07}
I.~Coiculescu and B.~Skorulski, \emph{Thermodynamic formalism of transcendental
  entire maps of finite singular type}, Monatsh. Math. \textbf{152} (2007),
  no.~2, 105--123.











\bibitem[DG87]{DG87}
R.~Devaney and L.~Goldberg, \emph{Uniformization of attracting basins for
  exponential maps}, Duke Math. J. \textbf{55} (1987), no.~2, 253--266.

\bibitem[DK84]{DK84}
R.~Devaney and M.~Krych, \emph{Dynamics of {${\rm exp}(z)$}}, Ergodic Theory
  Dynam. Systems \textbf{4} (1984), no.~1, 35--52.



	
  
	
	  \bibitem[Er92]{Er92}
  A. {\`E}.~Er{\"e}menko and and M. Yu. Lyubich, \emph{Dynamical properties of some classes of entire functions}, Ann. Inst. Fourier (Grenoble) \textbf{42}, (1992), no.~4, 989--1020. 
  
	
  \bibitem[Er89]{Er89}
  A. {\`E}.~Er{\"e}menko, \emph{On the iteration of entire functions}, Dynamical Systems and Ergodic Theory. Banach Center Publications, \textbf{23}, Warszawa, PWN, (1989), 339--345. 
  
 
  

   

    \bibitem[GK86]{GK86}
 Goldberg L.R. and L. Keen, \emph{A finiteness theorem for a dynamical class of entire functions},   Ergodic Theory and Dynamical Systems. 6, (1986), 183-192.
 
  

   
 

     
\bibitem[IS06]{IS06}
G.~Iommi and B.~Skorulski, \emph{Multifractal analysis for the exponential
  family}, Discrete Contin. Dyn. Syst. \textbf{16} (2006), no.~4, 857--869.
¿
\bibitem[Ka99]{Ka99}
B.~Karpi{\'n}ska, \emph{Hausdorff dimension of the hairs without endpoints for
  {$\lambda\exp z$}}, C. R. Acad. Sci. Paris S\'er. I Math. \textbf{328}
  (1999), no.~11, 1039--1044.
  
  
\bibitem[McM87]{McM87}
C. McMullen, \emph{Area and Hausdorff dimension of Julia sets of entire functions}, Trans. Amer. Math. Soc. 300 (1987), 329-342.
  
  





\bibitem[MU10]{MU10}
Volker Mayer and Mariusz Urba{\'n}ski, \emph{Thermodynamical formalism and
  multifractal analysis for meromorphic functions of finite order}, Mem. Amer.
  Math. Soc. \textbf{203} (2010), no.~954, vi+107. 
  
  \bibitem[MU01]{MU01}
D.~Mauldin and M.~Urbanski, \emph{Gibbs states on the symbolic space over an
  infinite alphabet}, Israel Journal of Mathematics \textbf{125} (2001), no.~1,
  93--130.
  
 
 
 
\bibitem[PU10]{PU10}
F.~Przytycki and U.~Urbanski, \emph{Conformal fractals-ergodic theory methods}, London Mathematical Society Lecture Note series, 371. Cambridge  University Press. Cambridge, 2010.
  20010.
  
  


\bibitem[Sa99]{Sa99}
O.~Sarig, \emph{Thermodynamic formalism for countable {M}arkov shifts}, Ergodic
  Theory Dynam. Systems \textbf{19} (1999), no.~6, 1565--1593.
  
%
%

\bibitem[St91]{St91}
G. M. Stallard, \emph{A class of meromorphic functions with no wandering domains}, Ann. Acad. Sei. Fenn. Ser. A I Math. 16 (1991), 211-226.



\bibitem[St91a]{St91a}
G. M. Stallard, \emph{The {H}ausdorff dimension of Julia sets of entire function}, 
 Ergodic Theory Dynamical Systems 11 (1991), 769-777


  


\bibitem[UZ04]{UZ04}
M.~Urba{\'n}ski and A.~Zdunik, \emph{Real analyticity of {H}ausdorff dimension of finer {J}ulia sets
  of exponential family}, Ergodic Theory Dynam. Systems \textbf{24} (2004),
  no.~1, 279--315.


\bibitem[UZ03]{UZ03}
M.~Urba{\'n}ski and A.~Zdunik, \emph{The finer geometry and dynamics of the
  hyperbolic exponential family}, Michigan Math. J. \textbf{51} (2003), no.~2,
  227--250.

\end{thebibliography}
\end{document}